\documentclass[11pt, english]{amsart}

\usepackage{verbatim}     
\usepackage{tikz}
\usetikzlibrary{calc,fadings,decorations.pathreplacing}
\usepackage{verbatim}   
\usepackage{graphicx}

\theoremstyle{plain}     
\newtheorem{thm}{Theorem}[section]  
    
\newtheorem{prop}[thm]{Proposition}      
\newtheorem{lemma}[thm]{Lemma}  
\newtheorem{coro}[thm]{Corollary}       
        
\newtheorem{dfn}[thm]{Definition}

\long\def\symbolfootnote[#1]#2{\begingroup%
\def\thefootnote{\fnsymbol{footnote}}\footnote[#1]{#2}\endgroup}   
 
\theoremstyle{remark}

\title{Expanding maps and continued fractions}

\thanks{Magee was supported in part by NSF Grant \#DMS-1128155. Oh was supported in part by NSF Grant \#1361673.}

\author{Michael Magee}
\address{School of Mathematics, Institute for Advanced Study, Princeton, NJ 08540}
\email{mmagee@ias.edu}
\author{Hee Oh}
\address{Mathematics department, Yale university, New Haven, CT 06511 and Korea Institute for Advanced Study, Seoul, Korea}
\email{hee.oh@yale.edu}
\author{Dale Winter}
\address{Department of Mathematics, Brown University, Providence, RI 02906}
\email{dale\_winter@brown.edu}

\begin{document}
\begin{abstract}
We obtain a power saving in the error term for a semigroup congruence lattice point count related to continued fractions. This is done by adapting arguments from recent work of Oh and Winter \cite{OW} that give uniform bounds for certain transfer operators in the congruence aspect. Our arguments also build crucially on work of Naud \cite{NAUD} and  Bourgain, Gamburd and Sarnak \cite{BGSACTA}. The result we obtain, together with a certain conjecture about the multiplicative combinatorics of $\mathrm{SL}_2(\mathbf{Z})$ that we highlight in the sequel, can be used to obtain an improvement on the size of the exceptional set  in Bourgain and Kontorovich's work  \cite{BKANNALS} on Zaremba's conjecture. 
\end{abstract}

\maketitle

\def\R{\mathbf{R}}
\def\C{\mathbf{C}}
\def\Z{\mathbf{Z}} 
\def\Q{\mathcal{Q}}
\def\RR{\mathcal{R}}

\def\a{\alpha}
\def\b{\beta}
\def\g{\gamma}
\def\d{\delta}
\def\e{\epsilon}
\def\i{\iota}
\def\G{\Gamma}
\def\GG{\mathbf{G}}
\def\vp{\varphi}

\def\GL{\mathrm{GL}}

\def\A{\mathcal{A}} 
\def\AA{\mathbb{A}} 
\def\B{\mathcal{B}} 
\def\E{\mathcal{E}}
\def\H{\mathbb{H}} 
\def\HH{\mathcal{H}}
\def\N{\mathcal{N}}
\def\O{\mathcal{O}} 
\def\Ohat{\widehat{\O}} 
\def\P{\mathcal{P}}
\def\K{\mathcal{K}}
\def\k{\kappa}
\def\T{\mathbb{T}}
\def\s{\sigma}
\def\spec{\mathrm{spec}}
\def\SO{\mathrm{SO}}
\def\End{\mathrm{End}}
\def\so{\mathfrak{so}} 
\def\SL{\mathrm{SL}}
\def\CC{\mathcal{C}}
\def\vol{\mathrm{vol}}
\def\Ind{\mathrm{Ind}}

\def\F{\mathbb{F}} 
\def\D{\mathcal{D}}
\def\L{\mathcal{L}}
\def\Mhat{\hat{\mathcal{M}}}
\def\diam{\mathrm{diam}}
\def\U{\mathcal{U}}
\def\Int{\mathrm{Int}}
\def\I{\mathcal{I}}
\def\Xhat{\widehat{X}}
\def\NN{\mathbb{N}}
\def\Lip{ {C^1}}
\def\ev{\mathrm{ev}}
\def\tr{\mathrm{tr}}

\tableofcontents

\section{Introduction}
Our initial set up is  that of expanding maps on Cantor sets as in Naud \cite{NAUD}. We consider $k\ge 2$ disjoint closed and bounded intervals $I_1 , \ldots , I_k \subset \R$ and set
\begin{equation}
I \equiv \cup_{j =1 }^k I_j .
\end{equation}
Our dynamics will come from a $C^2$ map 
\begin{equation}
T: I \to \R
\end{equation}
with the two properties
\begin{description}
\item[Eventually expanding]\label{expansion} There exist $\g >1$ and $D >0$ such that for all $N \geq 1$ and $x \in T^{-N +1}( I)$
\begin{equation}
| (T^N)'(x) | \geq D^{-1} \g^N .
\end{equation}
\item[Markov property] For all $i , j$, if $T(I_i) \cap \Int I_j \neq \emptyset$ then $T(I_i) \supset I_j$.
\end{description}
One can then reinterpret $T$ in the symbolic setting as follows. We define a $k \times k$ transition matrix $A$ by
\begin{equation}A_{i,j} =
\begin{cases}
    1 & \text{if $T(I_i) \supset I_j $}\\
    0 & \text{else}.
 \end{cases}
\end{equation}
We assume that there is a power $p_0>0$ with $A^{p_0} > 0$; this is the case when $T$ is topologically mixing.
We then obtain via $A$ a one sided subshift of finite type on the letters $\{ 1 , \ldots , k\}$ consisting of sequences that are permissible under $A$:
\begin{equation}
\Sigma_A^+ = \{ x=( x_n )_{n \in \mathbb{N}} \in \{ 1 , \ldots , k\}^\NN \: : \: A_{ x_{i} , x_{i+1} } =1\: \;\;  \forall i \in \NN  \}.
\end{equation}
In our key example of continued fractions this is the full shift with $A_{i j } = 1$ for all $i$ and $j$. The shift map $\sigma$ shifts sequences to the left so that $(\sigma x)_n = x_{n+1}$. We define
\begin{equation}
K = \bigcap_{i=0}^\infty T^{-i}(I) .
\end{equation}
Then $K$ is closed and invariant under $T$ and the dynamical systems $(K , T)$ and $(\Sigma_A^+ , \sigma)$ are topologically conjugate with each other.

We are also given a real valued function $\tau \in C^1(I)$ that is to be thought of as part of the dynamics\footnote{The dynamical system we will indirectly study is the suspension of the subshift specified by A, by the height function $\tau$.}.
We assume that $\tau$ is \textit{eventually positive}, which means that there is some $N\ge 1$ such that
\begin{equation}
\tau^N(x)  \equiv \sum_{ i = 0 }^{N - 1} \tau( T^i x ) 
\end{equation}
is strictly positive on $T^{-N}(I)$.
For $\mu$ a $T$-invariant probability measure on $K$, let $h_\mu(T)$ denote the entropy of $T$ with respect to $\mu$ .

The pressure functional for the function $-s \tau$ is defined by
\begin{equation}
P( -s \tau)  = \sup_ { \mu \in \mathcal{M}(K)^T } \left( h_\mu(T) -s \int_K \tau d\mu \right)
\end{equation}
where $ \mathcal{M}(K)^T$ denotes the set of all $T$-invariant probability measures on $K$. It follows
from the variational principle that $P( -s \tau)$ is strictly decreasing in $s$ and has a unique positive zero denoted by $s_0.$

Our functional analysis begins with the Banach space $C^1(I)$ equipped with the norm
\begin{equation}\label{eq:norms}
\| f \|_{C^1(I)} = \|f\|_\infty + \| f' \|_\infty.
\end{equation}
The complex transfer operator, defined for each $s \in \C$, is the bounded operator on $C^1(I)$ defined by
\begin{equation}
\L_{-s\tau}[f](x) = \sum_{ Ty = x } e^{-s\tau(y) } f(y) .
\end{equation}

We wish to extend these operators in an equivariant way to the space $C^1(I ; V)$ of continuously differentiable vector valued functions with norm analogous to \eqref{eq:norms}. Suppose we are given a locally constant assignment
\begin{equation}
c_0 : I \to \mathrm {SL}_2(\Z).
\end{equation}
We will view $c_0$ as fixed throughout this paper and we \textit{will always assume that the values $c_0(I_j)$, $1\le j\le k$, freely generate a semigroup, which we will denote by $\G$.}
For each $q \in \Z$, we define
\begin{equation}
\G_q \equiv \SL_2( \Z / q\Z ).
\end{equation}
By reducing $c_0$ modulo $q$, we get a map
\begin{equation}
c_q : I \to \G_q
\end{equation}
which via the right regular representation of $\G_q$ can be viewed as a unitary valued locally constant map
\begin{equation}
c_q : I \to U (  \C^{ \G_q } ).
\end{equation}
Moreover these $c_q$ have a compatibility condition for differing $q$, all having come from the same $c_0$. We define
\begin{equation}
 \L_{{-s\tau},q}[F](x) \equiv  \sum_{ Ty = x} e^{-s\tau(y) } c_q(y) . F(y) 
\end{equation}
which gives a bounded operator on $C^1( I ; \C^{\G_q } )$ for each $q$. We will sometimes use the shorthand $\L_{s, q}$ for this operator and we will frequently write $$s = a + ib.$$

We are going to assume that $\tau$ has a property called \textit{non local integrability} (NLI). This is a feature of the dynamical system coming from the pair  $(T , \tau)$ that goes beyond exponential mixing. We give the definition of (NLI) in Section \ref{largeimaginarypart} and we explain its relevance in our rough outline in Section \ref{rough}. This gives spectral estimates at $s$ with large imaginary part in Proposition \ref{bigimaginaryprop}.

For $s$ close to the real axis, spectral bounds come from a general mixing property that is (conjecturally) intrinsic to any free semigroup $\G \subset \SL_2(\Z)$ and the family of projections $\pi_q : \G \to \G_q$. In Definition  \ref{KSdfn} we define this property, which we label (MIX), for semigroups in $\SL_2(\Z)$ relative to some index set $\Q\subset \mathbb N$; this property encapsulates both the $\ell^2$-flattening Lemma and non-concentration phenomena for random walks in semigroups. 

It does not follow readily from the literature that this property holds in its strongest possible form that would allow $\Q = \mathbb{N}$. This statement would go beyond already formidable work of Bourgain and Varju in \cite{BVARBITRARY}. There, the type of general estimate that we require was circumvented by adroit application of work by Bourgain, Furman, Lindenstrauss and Mozes \cite{BFLM} on toral automorphisms.

We can now state our main Theorem.
\begin{thm}[Main Theorem]\label{mainthm}
Assume that $\tau \in C^1(I)$ has the non local integrability property (NLI) and that  $\G$ has the  combinatorial
mixing property (MIX) for some index set $\Q$. There is $Q_0 \in \mathbb{N}$ such that for any $\eta > 0$, there are $\e = \e(\eta)>0$, $b_0>0$, $0 <  \rho_\eta < 1$, $C_\eta>0$, $0 < \rho_0 < 1$, $r > 0$  and $C>0$ such that the following holds for all $a\in \R$ with $| a - s_0 | < \e$ and $b\in \R$:
\begin{enumerate}
\item \label{mainfirst} When $|b| \leq b_0$ and $f \in C^1(I ; \C^{\G_q} \ominus 1 )$
\begin{equation}
\| \L_{s,q}^m  f \|_{C^1} \leq C q^C  \rho_{0}^m \| f \|_{C^1}
\end{equation}
when  $ q \in \Q$ with $(q,Q_0) = 1$. Here $\C^{\G_q} \ominus 1$ is the orthocomplement to the trivial representation in the right regular representation of $\G_q$;
\item \label{mainsecond} When $| b | > b_0$ 
\begin{equation}
\| \L^m_{s,q}  \|_{C^1} \leq C_\eta | b |^{1+\eta } \rho^m_\eta
\end{equation}
uniformly with respect to $ q \in \Q$.
\end{enumerate}
 Here $s_0>0$ is the unique zero of $s \to P(-s\tau)$ and $P$ is the topological pressure.\end{thm}

A theorem parallel to this has recently been proved in the setting of convex cocompact subgroups of $\SL_2(\Z)$ by Oh and Winter \cite{OW}. This work resolved an outstanding question raised by the paper of Bourgain, Gamburd and Sarnak \cite{BGSACTA} as to whether the Selberg zeta functions of a tower of infinite volume congruence Riemann surfaces have a uniform zero free strip. Uniform exponential mixing results for the geodesic flow 
has also been deduced in \cite{OW} in the congruence aspect.  As is well-known (cf. \cite{MO}), such a result
yields a uniform power-savings error term in the associated congruence lattice point count (cf. Corollary \ref{maincoro} below).  Theorem \ref{mainthm} will  give a parallel proof of a uniform zero-free strip of the family of Selberg zeta functions. The translation from Theorem \ref{mainthm} to such a result is briefly explained in \cite{BGSACTA}, the relevant non local integrability property having been established by Naud \cite{NAUD}.

The proof of Part \ref{mainfirst} of Theorem \ref{mainthm} is a mild adaptation of work of Bourgain, Gamburd and Sarnak from \cite{BGSACTA}.
 The proof of Part \ref{mainsecond} transports the main observation of \cite{OW} that the non local integrability property is still effective in the vector valued setting of \cite{OW} to the current setting of expanding maps on the real line. All these ideas stem from the important work of Dolgopyat \cite{DOLG} that was translated to the setting under consideration by Naud \cite{NAUD}, building on work of Stoyanov \cite{STOYBILLIARD}. Indeed, in \cite{NAUD}, Naud proves the analog of Theorem \ref{mainthm} when there is no  congruence dependence.

We turn now to something more concrete. The reason we have chosen our particular setup is that there is an important application to a dynamical system associated to continued fractions and in particular to Zaremba's conjecture (\cite{ZAREMBA1}, \cite{ZAREMBA}). We have chosen to make this system the case study of this paper.

Let $\A \subset \mathbb{N}$ be any finite alphabet with at least two elements. We define the  semigroup ${\mathcal G}_\A$ to consist of all possible products of matrices 
\begin{equation} g_a \equiv 
\left(\begin{matrix}
0 & 1\\
1 & a 
\end{matrix}\right),\quad a \in \A
\end{equation}
in $\GL_2(\Z)$. We will pass to the (free) semigroup $\G_\A$ generated by matrices of the form
\begin{equation}
S_\A \equiv  \left\{
\left(\begin{matrix}
0 & 1\\
1 & a 
\end{matrix}\right) \cdot  \left(\begin{matrix}
0 & 1\\
1 & a' 
\end{matrix}\right)  : \quad a , a' \in \A \right\}.
\end{equation}
The family of semigroups  ${\mathcal G}_\A$  and $\G_\A$ was studied by Bourgain and Kontorovich in \cite{BKANNALS} in connection with continued fractions. 
If $x \in (0, 1) $ has continued fraction expansion 
\begin{equation}
x=  \cfrac{1}{a_1+\cfrac{1}{a_2+ \ddots \cfrac{1}{a_k + \ddots}}}
\end{equation} we write $x=[a_1, a_2, \ldots]$.
The numbers $a_j\in \mathbb N$ are called the partial quotients of $x$.
If $\Lambda(\A)$ denotes the set of all infinite continued fractions of the form
$[a_1 , a_2 , a_3 ,\ldots  ] $ where all $a_i \in \A$, it is not hard to see that if we have some fixed origin $o$ in the upper half plane $\H$, then $\Lambda(\A)$ coincides with the set of accumulation points of $\mathcal{G}_\A (o)$ in $\R$, where ${\mathcal G}_\A$ acts as M\"{o}bius transformations. This set is Cantor-like with associated Hausdorff dimension $0 < \d = \delta_\A < 1$.

We discuss lattice point count results for the semigroup $\G_\A$ which are uniform in the congruence aspect, with respect to the matrix norm
\begin{equation}
\left\|\left( \begin{matrix} a & b \\
c& d
\end{matrix}\right)  \right\| =\sqrt{  a^2  + b^2 + c^2 + d^2} .
\end{equation}

In Section \ref{continuedfractions}, we show that for each alphabet $\A$ there is a corresponding choice of $I$, $T$ and $\tau$ so that the periodic orbits of $T$ are related via $\tau$ to the traces of elements in $\G_\A$. Moreover under this correspondence the zero $s_0$ of the pressure functional $P(-s\tau)$ coincides with $\delta_\A$.

By work of Lalley \cite{LALLEYSYMB} and a certain effectivization that we give in Section \ref{continuedfractions}, it is possible to infer from our main Theorem \ref{mainthm} that one obtains a uniform power saving in the error term of the appropriate lattice point count.  To this end we show in Proposition \ref{distortionprop} that the distortion function $\tau$ associated to $\G_\A$ has the non local integrability (NLI) property. 

 We prove our counting result conditionally
  on the combinatorial mixing property  (MIX) for the $\G_{\A}$ stated in Definition \ref{KSdfn}.

\begin{coro}[Main Corollary]\label{maincoro}
If $\G_\A$ has the property (MIX) for $ \Q$, then there exist a finite modulus $Q_0 \in \mathbb{N}$, $C >0$ and $\e > 0$ such that for all $\g_0 \in \G_{\A}$,
 $\xi \in \SL_2(\Z / q \Z)$ and $q\in \Q$ with $(Q_0,q) = 1$,
\begin{equation}
\sum_{ \substack {    \frac{   \| \g \g_0 \| }{ \| \g_0 \| } \leq R     \\ \g \equiv \xi \bmod q         } } G( \g \g_0 o ) = \frac{ R^{2\delta_{\mathcal A}} }{ |\G_q | } C_*(\gamma_0, G\lvert_\R) + O \left(  (  \| G \|_\infty + \| [G\lvert_\R]' \|_\infty ) q^C  R^{  2(\delta_{\mathcal A} - \e)  }\right).
\end{equation}
Here $G$ is any function in $C^1 ( \H \cup \R)$ which is constant on some neighborhood of the cylinders of length $M$ in $\Lambda(\A)$ for some $M > 0$. The constant 
$C_*(\gamma_0, G\lvert_\R)>0$ is defined in detail in Section \ref{continuedfractions}. The implied constant depends on $M$.
\end{coro}

  We note here that the methods of this paper will yield that Corollary \ref{maincoro} holds for any free semigroup $\Gamma$ on at least 2 generators contained in a convex cocompact subgroup of $\SL_2(\Z)$.

The result of Bourgain, Gamburd and Sarnak   \cite{BGSACTA}
referenced  in \cite[Theorem 8.1]{BKANNALS}  does not have such a uniform error term. Their saving is of the form $
R^{2\delta_{\mathcal A} - c / \log \log R }$, under the same hypothesis as in Corollary \ref{maincoro}.

In Section \ref{modular} Lemma \ref{sfmix}, we show that 
if $\G \subset \SL_2(\Z)$ is any semigroup that is freely generated by $c_0(I)$, then $\G$ has property (MIX) for $\Q$ consisting of square free numbers. This is a relatively minor extension  of prior work of  Bourgain, Gamburd and Sarnak \cite{BGSACTA} following in turn  from work of Bourgain and Gamburd \cite{BOURGAINGAMBURDEXP}. The only novelty here is that we are not necessarily contained inside a free group, but we are inside the free product $\SL_2(\Z)$. Lemma \ref{sfmix} thus gives
\begin{coro}
Corollary \ref{maincoro} holds when $\Q$ consists of square free numbers.
\end{coro}

The first named author (Magee) expects to extend the mixing property (MIX) for semigroups in $\SL_2(\Z)$ to $\Q = \mathbb{N}$ (avoiding maybe some finite modulus) in future work.

Our refinement of this counting estimate is an important step towards a power-savings estimate on the exceptional set for Zaremba's conjecture. 
Let $\mathfrak D_{\A}$ denote the set of all positive integers $d$ such that for some $b$ coprime to $d$, all the partial quotients of $b/d$ belong to $\A$. 
Zaremba's conjecture (\cite{ZAREMBA1}, \cite{ZAREMBA}) states that $$\mathfrak D_{\{1, 2, \cdots, A\}} = \mathbb{N}$$ for some sufficiently large $A>0$; in fact,
$A=5$ should suffice. Zaremba's conjecture was motivated by its applications to numerical integration, Monte-Carlo methods
and pseudo-random numbers (cf. \cite{BOURGAINSURVEY}).
Bourgain and Kontorovich \cite{BKANNALS} obtained the following estimate
\begin{equation}\label{bk} \# \mathfrak D_{\{1, 2, \cdots, 50\}}\cap [1, N] = N+O(N^{1-c/(\log \log N)})\end{equation}
based on the aforementioned counting result of Bourgain, Gamburd and Sarnak \cite{BGSACTA}, so implicitly conditional
on the (MIX) property of $\G_{\{1,2,\cdots, 50\}}$ for $\mathbb N$. The qualitative
density zero statetment $\# \mathfrak D_{\{1, 2, \cdots, 50\}}\cap [1, N] = N+o(N)$ was proved unconditionally
in the arXiv version of their paper (arXiv:1107.3776v1). Recent work of Huang \cite{Huang} established similar estimates for the alphabet $\mathcal{A} := \lbrace 1, 2,3,4 , 5 \rbrace$ by improving estimates on multilinear sums using Vinogradov's method.

We remark that the relevance of the semigroup-counting problem for $\mathcal G_\A$ (or to $\G_\A$)
to Zaremba's conjecture  comes from the relation  $$\mathfrak D_{\A}=\langle \mathcal G_\A \begin{pmatrix} 0 \\ 1 \end{pmatrix},
\begin{pmatrix} 0 \\ 1 \end{pmatrix}\rangle .$$

Combining Bourgain-Kontorovich's method, Huang's refinement, and with the counting estimate Corollary \ref{maincoro} in place of \cite[Theorem 8.1]{BKANNALS}, we obtain the following improvement of \eqref{bk} with a power-savings error term
(the key point is that this enables us to replace the parameter $\mathcal Q=N^{\alpha_0/\log\log N}$ in \cite{BKANNALS} and \cite{Huang} with a power of $N$).

\begin{thm}\label{mainZa}
Letting $\mathcal A=\{ 1,2,3,4, 5\}$, assume that $\Gamma_{\mathcal A}$ has the (MIX) property with respect to $\Q=\mathbb N$. Then for some $\e>0$,
$$\#\mathfrak D_{\A}\cap [1, N]=N +O(N^{1-\epsilon}).$$
\end{thm}
We note that, in a  recent paper \cite{BOURGAINPARTIAL}, Bourgain has sketched an alternative proof of Theorem \ref{mainZa}, again under the hypothesis of the (MIX) property.
As in \cite{BKANNALS}, Theorem \ref{mainZa} holds for any alphabet $\mathcal A$, provided
the dimension $\delta_\A$ is sufficiently large and
there is no local obstruction for $\mathfrak D_\A$.
More precisely, setting
$$\mathfrak A_{\mathcal A}:=\{d\in \mathbb N: d\in \mathfrak D_{\mathcal A} (\text{mod } q) \text{ for all $q>1$}\},
$$
we have:
\begin{thm} \label{zaremba2} Let $\delta_0=5/6$. For any finite alphabet $\mathcal A$ with $\delta_{\mathcal A}>\delta_0$,
there exists $\e>0$ such that as $N\to \infty$,
$$\frac{\# (\mathfrak D_{\mathcal A}\cap [N/2, N])}{\# (\mathfrak A_{\mathcal A}\cap [N/2, N])}= 1+O(N^{-\e}),$$
provided $\Gamma_{\mathcal A}$ has the (MIX) property with respect to $\Q=\mathbb N$. \end{thm}
We note that  $\mathfrak A_{\A}=\mathbb Z$ if $\A$ contains consecutive numbers $m, m+1$ for some $m$,
and $\delta_{\{1,2,,3,4, 5\}}>5/6$ by Jenkinson \cite{Jenkinson}; for general $\mathcal{A}$ see also the algorithm of Jenkinson and Pollicott \cite{JP}.
Also note that $$\delta_{\{1, 2, \cdots, A\}}=1-\frac{6}{\pi^2 A}-\frac{72\log A}{\pi^4 A^2} +O\left(\frac{1}{A^2}\right)$$ was given by Hensley \cite{Hensley}.
 
  We also draw the reader's attention to the survey article \cite{BOURGAINSURVEY} where other applications to continued fractions are discussed. The lattice point count of Corollary \ref{maincoro} is also related to the work of Bourgain and Kontorovich \cite{BKBEYONDGEODESICS} on low lying fundamental geodesics on the modular curve.



\subsection{Ideas of proof}\label{rough}
In this section we give a suggestive proof, without any rigor, that the ideas of this paper work. We expect that the reader will find this much more useful than the technical details we will later give. Our strategy will be to create a toy dynamical system, explain how the ideas of Dolgopyat function here, and comment upon the passage to the equivariant setting of this paper. The notation we will use here is independent of the rest of the paper. Our discussion derives from that of Dolgopyat in \cite[Section 5]{DOLG}.

Consider dynamics of the form
\begin{equation}
\phi : \R \times K \to K 
\end{equation}
that corresponds to the action of $\R$ on compact metric and measure space $K$. Suppose we are also given a function $\tau : K \to \R$ which encodes that if we are at $x \in K$, our flow should move us forwards $\tau(x)$ in time. In other words we now have discrete time dynamics
\begin{equation}
\Phi : K \to K , \quad \Phi( x )\equiv  \phi( \tau(x) , x) .
\end{equation}
One is motivated, for various reasons\footnote{In our case, the transfer operator appears in the Laplace transform of the renewal equation (see Section \ref{renewalsection}).}\footnote{Another arises in the setting of Anosov flows where by a procedure of Dolgopyat from \cite{DOLG}, contraction properties of these operators are related to the decay of correlations of $\phi$.}, to study operators roughly of the form
\begin{equation}
L_{\theta} [h] (x) \equiv \sum_{ y : \: \Phi(y) = x }e^{i\theta\tau(y)}  h(y). 
\end{equation}
One is interested in large $\theta$ behavior of $L_\theta$. This is analogous, in light of well known connections to dynamical zeta functions,  to consideration of number-theoretic $L$-functions and in particular Riemann's zeta function, close to the edge of the critical strip at large imaginary values. The goal is to control large powers  $L_\theta^n$ in various operator norms that we will gloss over for now.

Let us consider briefly $\tau \equiv 1$. Then we are considering dynamics that arise from iterating\footnote{One key case here is that of Anosov diffeomorphisms.} the map $\Phi = \phi(1 , \bullet)$. One calculates
\begin{equation}
L^n_{\theta} [h] (x) = e^{i n \theta} \sum_{ y : \:\Phi^n(y) = x} h(y) .
\end{equation}
If one assumes \textit{that $\Phi$ is exponentially mixing}, then the sum over past histories for $x$ becomes very well distributed in $n$, so that $L^n_\theta$ tends to rapidly flatten functions, hence one can achieve the desired operator norm bounds. The $\theta$ dependence is irrelevant here.

Now we incorporate non constant $\tau$ into our discussion.
In the case of general $\tau$ one has
\begin{equation}
L^n_\theta[h] (x) = \sum_{ y  : \: \Phi^n(y) = x } e^{i \theta ( \tau(y) + \tau(\Phi(y) ) + \ldots + \tau ( \Phi^{n-1}(y) ) } h(y).
\end{equation}
Notice here that the summands now depend on the entire trajectory $y , \Phi y , \Phi^2 y$ etc. One would like to use exponential mixing properties of $\Phi$ but it is no longer obvious how to do so. 
Clearly the function
\begin{equation}
\tau^n(y) \equiv  \tau(y) + \tau(\Phi(y) ) + \ldots + \tau ( \Phi^{n-1}(y) )
\end{equation}
is going to feature in the subsequent analysis of $L_\theta$. Let us try to analyze\footnote{This $L^2$ analysis is converted into other norms by standard estimates.} the $L^2$ norm of $L_\theta^n h$. One gets formally
\begin{equation}
\int_K | L_\theta^n h |^2 dK = \int_{ K } \left( \sum_{  y_1 , y_2  : \: \Phi^n(y_i) = x  , \: i =1,2} e^{i \theta [ \tau^n(y_1) - \tau^n(y_2)] }  h(y_1) h (y_2)   \right) dK(x) .
\end{equation}
This can be converted into a sum over branches of $\Phi^{-n}$. Suppose there are finitely many functions $\a^n$ such that $\Phi^n \circ \a^n$ is the identity on $K$. Then interchanging the sum one gets (supposing $\Phi$ measure preserving)
\begin{equation}
\| L_\theta^n h \|_{L^2(K)}^2 = \sum_{ \a_1^n , \a_2^n  } \int_{K} e^{i \theta [ \tau^n(\a^n_1(k) ) - \tau^n(\a^n_2(k))  ] }h(\a^n_1(k) ) h (\a^n_2(k) )  dK(k) .
\end{equation}
The contribution from $\a_1^n$ and $\a_2^n$ that contain pairs of trajectories that remain close through the first $n$ iterates of $\Phi$ is controlled by exponential mixing, by estimating in absolute values. Note that this argument would work just fine if $h$ was a vector valued function. It remains to control contributions from $\a_1^n, \a_2^n$ for which all trajectories are separated, that is, the off-diagonal terms.

The ingredient that is to be provided, beyond exponential mixing, is that for example, it could be the case that away from the diagonal all of the functions
\begin{equation}\label{eq:branchs}
[\tau^n \circ \a_1^n - \tau^n \circ \a_2^n ]
\end{equation}
have derivatives that are bounded away from $0$. For large $\theta$, this would make estimation of
\begin{equation}
\int_{k \in K} e^{i \theta [ \tau^n(\a^n_1(k) ) - \tau^n(\a^n_2(k))  ] }h(\a^n_1(k) ) h (\a^n_2(k) )  dK(k)
\end{equation}
amenable to non-stationary phase considerations. This is  very loosely the case in our setting, with some very important qualifiers. The set $K$ is actually a Cantor set inside $\R$, which means that derivative considerations are still available, but the set being totally disconnected brings technical difficulties. One other obstacle is that the non local integrability assumption that we make, following Naud, only implies that we have a non vanishing derivative for two particular inverse branches (cf. Lemma \ref{NLIcons}). Cancellation in the full off-diagonal is achieved by a rather technical induction involving certain Dolgopyat operators. All of the arguments of Section \ref{operators} are to this end, which can be thought of as a very elaborated version of the Van der Corput Lemma.

We have now completed our description of the methods of Dolgopyat and Naud. The purpose of this paper is to extend these results to transfer operators that have been twisted by unitary operators on a finite dimensional vector space. These twists lie in a family throughout which we wish to seek uniform bounds.

So we now turn to the extension of our previous discussion regarding large $\theta$ to vector valued functions. Our twisted transfer operator will be of the form
\begin{equation}
L_{\theta,u }[h] (x) = \sum_{ y : \: \Phi(y) = x }e^{i\theta\tau(y)}  u(y) h(y) 
\end{equation}
which now acts on $V$-valued functions, with $V$ a finite dimensional complex vector space and $u : K \to U(V)$ a unitary valued function. We treat locally constant $u$. As we have remarked before, control of near-diagonal terms appearing in $\| L^n_{\theta , u}\|_{L^2(K ;V)}$ by exponential mixing still works.

The off-diagonal contribution to the (square of the) $L^2$ norm of $L^n_{\theta , u} h$ is again controlled by Van der Corput type arguments. The extension of this analysis to vector valued functions, in the setting considered by Naud \cite{NAUD}, is one new contribution of this paper. It turns out that the inductive procedure of Dolgopyat can be essentially unaltered, and we have pushed down the vector valued considerations to Lemma \ref{technical} where we do most of our work. We need, critically, to use that $u$ is locally constant so that it does not interfere with the previously known oscillations. This is analogous to dealing with flat vector bundles, that are defined by locally constant cocycles.

\subsection{Acknowledgements}
We would like to thank Peter Sarnak for his encouragement and support throughout this project. We would also like to thank Alex Kontorovich for the pointer to the work of Bourgain in \cite{BOURGAINPARTIAL}. We thank Jean Bourgain and Curt McMullen for helpful comments on an earlier
version of this paper.
\section{Bounds for transfer operators: large imaginary part}\label{largeimaginarypart} 

 Recall from the introduction the set $I$, $K$, the map $T: I \to \R$, the transition matrix $A$, the cocycle $c_0$ and $\G$. We keep these notations for Sections 2-4.
Let $\Sigma_A^{-}$ be the space of negatively indexed $A$-permissible sequences on $\{1 ,\ldots , k \}.$ We define for $\xi  \in \Sigma_A^{-}$ the function
\begin{equation}\label{eq:Delta}
\Delta_\xi( u ,v) = \sum_{ i = 0}^{\infty} \tau( T^{-1}_{\xi_{-i}} \circ \ldots \circ T^{-1}_{\xi_0} u ) -  \tau( T^{-1}_{\xi_{-i}} \circ \ldots \circ T^{-1}_{\xi_0} v )
\end{equation}
on $I_j \times I_j$ such that $T(I_{\xi_0}) \supset I_j$. It follows from the expanding property of $T$ that $\Delta_{\xi}$ is $C^1$ where it is defined. Naud (following others) has defined a temporal distance function
\begin{equation}\label{eq:temporaldistance}
\varphi_{\xi , \eta}(u,v)  = \Delta_\xi(u,v) - \Delta_\eta(u,v)
\end{equation}
which is defined for each $\xi , \eta \in \Sigma_A^{-}$ and $u , v \in I_j$ such that $T(I_{\xi_0}) \cap T(I_{\eta_0}) \supset I_j$. We treat $\tau$ with the following important property.

\begin{dfn}[Non local integrability (NLI)]\label{NLI}
An eventually positive function $\tau$ has property (NLI) if there is $j_0 \in \{ 1, \ldots , k\}$, $\xi , \eta \in \Sigma_A^-$ with $T(I_{\xi_0}) \cap T( I _{ \eta_0 } ) \supset I_{j_0}$ and $u_0 , v_0 \in K \cap I_{j_0}$ such that
\begin{equation}
\frac{ \partial \varphi_{\xi , \eta }}{\partial u } (u_0 , v_0) \neq 0.
\end{equation}
\end{dfn}

One novelty of this paper is the following uniform version of \cite[Theorem 2.3]{NAUD}.
\begin{prop}\label{bigimaginaryprop}
There is $b_0>0$ such that part \ref{mainsecond} of Theorem \ref{mainthm} holds. That is, for any $\eta >0$, there is $0 < \rho_\eta < 1$ such that
\begin{equation}
\| \L^m_{s,q}  \|_{C^1} \ll_\eta | b |^{1+\eta } \rho^m_\eta
\end{equation}
when $|b| > b_0$ and  $q \in \Q$, under the same hypotheses\footnote{No hypotheses on $\Q$ are really need here - in fact the $c_q$ could be any locally constant unitary valued map.} as Theorem \ref{mainthm}.
\end{prop}
We now show how to relate this Proposition to the construction of certain Dolgopyat operators.
Firstly we need the following Theorem. We state this following Naud \cite{NAUD} and the result can also be found in \cite{PPAST}.
\begin{thm}[Ruelle-Perron-Frobenius] \label{RPFTheorem}
Let $L(K)$ denote the Banach space of Lipschitz functions on $K$. For real valued $f \in L(K)$ let $\L_f$ act on $L(K)$ by 
\begin{equation}
\L_f [ G ](x)  = \sum_{ T y = x } e^{ f(y)} G(y) .
\end{equation}
Then
\begin{enumerate}
\item There is a unique probability measure $\nu_f$ on $K$ such that $\L_f^*(\nu_f) = e^{P(f) } \nu_f$.
\item The maximal eigenvalue of $\L_f$ is $e^{P(f)}$ which belongs to a unique positive eigenfunction $h_f \in L(K)$ with $\nu_f(h_f) = 1$.
\item The remainder of the spectrum of $\L_f$ is contained in a disc of radius strictly less than $e^{P(f)}$.
\end{enumerate}
\end{thm}
As in \cite{NAUD} we need to note that this Theorem extends reasonably to $\L_f$ acting on $C^1(I)$ given $f \in C^1(I)$. In particular $\L_f$ acting on $C^1(I)$ has the same spectral properties relative to a positive eigenfunction $h_f \in C^1(I)$ such that $\L_f h_f = e^{P(f) }h_f$. We also view $\nu_f$ as a measure on $C^1(I)$ with support in $K$.

Let $h_a$ be the normalized positive eigenfunction  of $\L_{-a\tau}$ corresponding to the maximal eigenvalue $\exp( P ( - a\tau ))$. We set
\begin{equation}
\tau_a = -a\tau - P(-a\tau) - \log( h_a \circ T) + \log(h_a) .
\end{equation}
We now renormalize our transfer operators by defining
\begin{equation}
L_{s,q} \equiv \L_{ \tau_a - ib\tau , q} .
\end{equation}
This is the same as
\begin{equation}\label{eq:perturbed}
L_{s,q} = \exp ( - P( -a\tau ) ) M_{h_a}^{-1} \L_{s,q} M_{h_a}
\end{equation}
where $M_{h_a}$ is multiplication by $h_a$. It now follows by arguments as in Naud \cite[pg. 132]{NAUD} that it is enough to prove Proposition \ref{bigimaginaryprop} and Theorem \ref{mainthm} with $L_{s,q}$ in place of $\L_{s,q}$. We also note here that the maximal eigenfunction of $L_a$ is the constant function, with eigenvalue $1$, that is $L_a 1 = 1$ for $a \in \R$.

The rest of the passage to the estimates in the next section is routine but we give some of the details for completeness. One shows that in order to prove Proposition \ref{bigimaginaryprop} it is enough to prove
\begin{lemma}\label{powers}
With the same conditions as Theorem \ref{mainthm}, there are $N >0$ and $\rho \in (0,1)$ such that when $|a - s_0|$ is sufficiently small and $|b|$ is sufficiently large we have
\begin{equation}
\int_{K} | L^{n N }_{s,q} W |^2 d  \nu_0 \leq \rho^n ,
\end{equation}
where $W \in C^1(I ; \C^{\G_q  } )$, $d\nu_0 = h_{ - \delta \tau} \nu_{-\delta \tau}$ is the Gibbs measure on $K$, and $\| W \|_{(b) } \leq 1$, which stands for the warped Sobolev norm
\begin{equation}
\| W \|_{(b) }  = \| W \|_\infty + |b|^{-1} \| W' \|_\infty.
\end{equation}
These estimates are uniform in $q$.
\end{lemma}
This corresponds to \cite[Theorem 3.1]{OW} in the work of Oh and Winter and is the uniform version of \cite[Proposition 5.3]{NAUD}. 

Lemma \ref{powers} implies  Proposition \ref{bigimaginaryprop} by the use of a priori estimates for the transfer operators that allow one to convert an $L^2$ estimate into a $C^1$ bound. These estimates are given in \cite[Lemma 5.2]{NAUD} for complex valued functions. They are however easily proved for vector valued functions giving
\begin{lemma}\label{apriori} 
There are $\k_1 , \k_2 , a_0 , b_0 >0$ and $R < 1$ such that for $|a - s_0| <a_0$ and $|b | > b_0$ we have for all $f \in C^1(I ; \C^{\G_q})$
\begin{equation}\label{eq:apriori1}
\| [L^n_{s , q} f]' \|_\infty \leq \k_1 | b | \| L^n_a f \|_\infty + R^n \| L^n_a |f'| \|_\infty ,
\end{equation}
and
\begin{equation}\label{eq:apriori2}
\| L^n_{s_0 , q } f \|_\infty \leq \int_K |f| d  \nu_0 + \k_2 R^n \| f \|_{L(K) } .
\end{equation}
\end{lemma}
Lemma \ref{apriori} together with Lemma \ref{powers} imply Proposition \ref{bigimaginaryprop} by arguments  appearing in \cite[pp. 133-134]{NAUD}. Roughly speaking the ingredients are Cauchy-Schwarz to access Lemma \ref{powers}, remarks regarding the behaviour of $\tau_a^m$ for $a$ close to $s_0$ that appear elsewhere in this paper, and splitting up exponents in the form $m = nN + r$.

The proof of Lemma \ref{powers} proceeds through the construction of certain Dolgopyat operators that we give in the next section. 
\section{Construction of uniform Dolgopyat operators}\label{operators}
We follow the notation of Naud \cite{NAUD}. For $A >0$ we consider the cone
\begin{equation}
\CC_A \equiv \{ H \in C^1(I) : H > 0 \text{ and } |H'(x)| \leq A H(x) \: \text{ for all } x \in I \:\}.
\end{equation}
In this section we establish a uniform version of the key Lemma of Naud \cite[Lemma 5.4]{NAUD}. This is also analogous to \cite[Theorem 3.3]{OW}.
\begin{lemma}[Construction of uniform Dolgopyat operators]\label{core}
Suppose $\tau$ has the (NLI) property. There exists $N > 0 , $ $A > 1$ and $\rho \in  (0,1)$ such that for all $s= a+ib$ with $|a - s_0|$ small and $|b| > b_0$ large, there exists a finite set of operators $(\N_s^J)_{J \in \E_s }$ that are bounded on $C^1(I)$ and satisfy the following three conditions
\begin{enumerate}
\item \label{corefirst} The cone $\CC_{A|b|}$ is stable by $\N^J_s$ for all $J \in \E_s$.
\item \label{coresecond} For all $H \in \CC_{A|b|}$ and all $J \in \E_s$,
\begin{equation}
\int_K |\N^J_s H |^2 d \nu_0 \leq \rho \int_K | H |^2 d \nu_0 .
\end{equation}
\item \label{corethird} Given $H \in \CC_{A|b| }$ and $f \in C^1( I ; \C^{\G_q} )$ such that $|f| \leq H$ and $|f'| \leq A|b|H$, there is $J \in \E_s$ with
\begin{equation}
| L^N_{s,q} f | \leq \N_s^J H , \quad \text{and} \: \: |  ( L^N_{s,q} f )' | \leq A|b| \N^J_s H .
\end{equation}
\end{enumerate}
\end{lemma}
When we write $|f|$ for $f \in C^1( I ; \C^{\G_q} )$ we refer to the function obtained by taking pointwise Euclidean ($l^2$) norms.
We now show that the existence of these operators implies Lemma \ref{powers}.
\begin{proof}[Lemma \ref{core} implies Lemma \ref{powers}] 

Given this construction (Lemma \ref{core}), Lemma \ref{powers} is proved following the argument of \cite[pg. 21]{OW} or one in \cite[pg. 135]{NAUD}.
Indeed given non zero $f \in C^1(I ; \C^{\G_q})$ with $\| f \|_{(b)} \leq 1$ (cf. Lemma \ref{powers} for the definition of $\| \|_{(b) }$), we define 
\begin{equation}
H = \| f \|_{(b)} 1.
\end{equation}
One sees that $H$ and $f$ are as in Lemma \ref{core}, that is,  $H \in \CC_{A|b| }$, $|f| \leq H$,  and  $| f'| \leq A|b| H$ as $A > 1$. One gets then by part \ref{corethird} of Lemma \ref{core} that
\begin{equation}
| L^N_{s,q} f | \leq \N^J_s H ,\quad | (L^N_{s,q} f)' | \leq A|b| \N^J_s H
\end{equation} 
for some $J \in \E_s$. Since $\CC_{A|b|}$ is stable under the $N^J_s$ one can repeat this to get for some sequence $J_1 , \ldots , J_n \in \E_s$ that
\begin{equation}
\int_K | L^{nN}_s f |^2 d\nu_0 \leq \int_K | \N^{J_N}_s \ldots \N^{J_1}_s H |^2 d\nu_0 \leq \rho^n \int_K |H|^2 d\nu_0 \leq \rho^n 
\end{equation}
by using part \ref{coresecond} of Lemma \ref{core}.
\end{proof}

The first two properties of Lemma \ref{core} were proved by Naud in \cite{NAUD}; we follow closely Naud's construction of the operators in the following.

\subsection{Consequences of non local integrability (NLI)}

Naud notes the following consequence of (NLI) that we will use later.
\begin{lemma}[Proposition 5.5 of \cite{NAUD}]\label{NLIcons}
If $\tau$ has property (NLI), there are $m , m',  N_0 >0$ such that for all $N > N_0$, there are two branches $\a_1^N, \a_2^N$ of $T^{-N}$ with
\begin{equation}
m' \geq \left| \frac{d}{du} [ \tau^N \circ \a_1^N  - \tau^N \circ \a_2^N ]( u) \right| \geq m >0,  \quad \forall u \in I.
\end{equation}
\end{lemma}
We remark here that the lower bound is the harder one. The upper bound follows from the expanding property of $T$ and regularity of $\tau$.

Now suppose we deal with $\tau$ with property (NLI). Let $\xi , \eta, u_0, v_0$ and $j_0$ be as in Definition \ref{NLI}. \textit{Throughout the rest of this paper, the assignments $N \to \a_1^N$ and $N \to \a_2^N$ are fixed as those given by Lemma \ref{NLIcons}. } We do however need to know some of the details about how the $\a_i^N$ have been constructed, which we give now.

As in the proof of \cite[Proposition 5.5]{NAUD} there are $\e > 0$ and an open interval  $\U$ with
\begin{equation}
I_{j_0} \supset \U \ni u_0 
\end{equation}
such that
\begin{equation}
\left| \frac{ \partial \varphi_{\xi , \eta }}{\partial u } (u' , v_0) \right| > \e
\end{equation}
for all $u' \in \U$.
We define for any $n$
\begin{equation}
\beta_1^n = T^{-1}_{\xi_{-n+1}} \circ \ldots \circ T_{\xi_0}^{-1},
\end{equation}
\begin{equation}
\beta_2^n = T^{-1}_{\eta_{-n+1}} \circ \ldots \circ T_{\eta_0}^{-1},
\end{equation}
two branches of $T^{-n}$ on $I_{j_0}$. In the proof of \cite[Proposition 5.5]{NAUD}, Naud also constructs \begin{equation}
\psi : I \to \U 
\end{equation}
which is a branch of $T^{- \hat{p} }$ for some $\hat{p}$ a fixed positive integer related to the mixing and expanding properties of $T$. The image of $\psi$ is a disjoint union of $k$ closed intervals each of which is diffeomorphic to some $I_j$ by $\psi$. We denote by $U_0$ the image of $\psi$.
 We will use the parameterization
\begin{equation}
N = \tilde{N } + \hat{p}, 
\end{equation}
Then the $\a_i^N$ are defined by
\begin{equation}
\a_i^N = \beta_i^{\tilde{N} } \circ \psi .
\end{equation}
As $\tilde{p}$ is fixed, $\tilde{N}$ and $N$ are coupled. They are to be chosen, depending on $b$ and other demands in the following.

\subsection{Construction of Dolgopyat operators}
The following is proved by Naud \cite[Proposition 5.6]{NAUD}.

\begin{prop}[Triadic partition] \label{triad}
There are $A_1 , A'_1 > 0$ and $A_2 > 0$ such that when $\e > 0$ is small enough, there is a finite collection $(V_i)_{1 \leq i \leq Q}$ of closed intervals ordered along $U_0$ such that:
\begin{enumerate}
\item $\U \supset \cup_{i =1}^Q V_i \supset U_0$, $V_i \cap \Int U_0 \neq \emptyset$ for all $i$ and $\Int V_i \cap \Int V_j = \emptyset$ when $i \neq j$.
\item  \label{second} For all $1 \leq i \leq Q$, $\e A'_1 \leq | V _i | \leq \e A_1 $.
\item  For all $1 \leq j \leq Q$ with $V_j \cap K \neq \emptyset$, either $V_{j-1} \cap K \neq \emptyset$ and $V_{j+1} \cap K \neq \emptyset $ or  $V_{j-2} \cap K \neq \emptyset$ and $V_{j-1} \cap K \neq \emptyset $ or  $V_{j+1} \cap K \neq \emptyset$ and $V_{j+2} \cap K \neq \emptyset $  . In other words, intervals that intersect $K$ come at  least in triads.
\item \label{fourth} For all $1 \leq i \leq Q$ with $V_i \cap K \neq \emptyset$, $V_i \cap K \subset U_0$ and $\mathrm{dist}( \partial V_i, K) \geq A_2|V_i|$.
\end{enumerate}
\end{prop}
Now following Naud we can construct the Dolgopyat operators. Suppose that we are working at frequency $s = a + ib$. Then for fixed $\e'$ to be chosen, we construct a triadic partition $(V_i)^Q_{i =1}$ of $U_0$ with $\e = \e' / |b|$ as in Proposition \ref{triad}. Then for all $i \in \{ 1, 2 \}$ and $j \in \{1 ,\ldots,Q\}$ we set
\begin{equation}
Z^i_j = \beta_i^{\tilde{N}} ( V_j \cap U_0 ).
\end{equation} 
We will write
\begin{equation}
X_j = \{ x \in I : \psi(x) \in V_j \} , \quad 1 \leq j \leq Q .
\end{equation}
Properties  \ref{fourth} and \ref{second} of Proposition \ref{triad} imply that 
\begin{equation}\label{eq:cutoffest}
\mathrm{dist}( K \cap V_j , \partial V_j ) \geq A_2|V_j| \geq  \frac{A_2 A'_1 \e' } { |b| } .
\end{equation}
whenever $K \cap V_j \neq 0$. For such $j$ we can find a $C^1$ cutoff $\chi_j$ on $I$ that is $\equiv 1$ on the convex hull of $K \cap V_j$ and $\equiv 0$ outside $V_j$. Due to \eqref{eq:cutoffest} we can ensure that
\begin{equation}
| \chi'_j | \leq A_3 \frac{|b|}{\e'} ,\quad A_3 = A_3( A_2 , A'_1 ). 
\end{equation}
Then the index set $\I_s$ is defined to be
\begin{equation}
\I_s \equiv \{ ( i ,j ) \: : \: 1 \leq i \leq 2 , 1 \leq j \leq Q , V_j \cap K \neq \emptyset \}. 
\end{equation}
Allow $0 < \theta < 1$ to be fixed shortly. For all $J \subset \I_s$ we define $\chi_J \in C^1(I)$ by
\begin{equation}
\chi_J(x)  = \begin{cases}
    1 - \theta \chi_j( \psi(T^N x) ), & \text{if $x \in Z^j_i$ for $(i,j) \in J $}.\\
    1, & \text{else}.
  \end{cases}
\end{equation}
Then the Dolgopyat operators on $C^1(I)$ are defined by
\begin{equation}
\N^J_s (f) = L^N_a ( \chi_J f ) .
\end{equation}
Recall that $L_a$ is the transfer operator at $s = a$. 

Let us return to our Lemma \ref{core} so that we can complete our definitions. 

\begin{dfn}\label{dense}
We say that $J \subset I_s$ is dense if for all $1 \leq j \leq Q$ with $V_j \cap K \neq \emptyset$ 
there is some $1 \leq j' \leq Q$ with $(i , j') \in J$ for some $i \in \{ 1, 2 \}$ and with $|j - j'| \leq 2.$ 
\end{dfn}
We define $\E_s$ of Lemma \ref{core} to be the set of $J \subset \I_s$ such that $J$ is dense.

The following is proved in \cite{NAUD} -  we have tried to contain everything that we use as a black box here.
\begin{prop}[Naud] \label{blackbox}
When $N$ is large enough and $\theta$, $\e'$  are small enough, there are $A$ and $\rho$ as in Lemma \ref{core}  such that properties \ref{corefirst} and \ref{coresecond} hold for our $(N , |b| , \theta , \e')$ parameterized and $\E_s$-indexed Dolgopyat operators whenever $|b|$ is large enough and $|a - s_0|$ is small enough. For the same range of  parameters
we can also ensure
\begin{equation}
 |([ \tau^N_a + i b \tau^N ] \circ \a^N )'(x)| \leq  \frac{ A |b| }{4}\end{equation}
for any inverse branch $\alpha^N$ of $T^N$. 
\end{prop}
The proof of the inequality above is on \cite[pg. 137]{NAUD}.

\textit{This fully completes the definition of the Dolgopyat operators modulo choice of $\e'$, $ \theta $ and $N$ - the $A$ and $\rho$ required for Lemma \ref{core} are that specified by Lemma \ref{blackbox} given these parameters.}

\subsection{Proof of Lemma \ref{core}, property \ref{corethird}}
Our remaining task in this section is to prove property \ref{corethird} of Lemma \ref{core}. This is proved for complex valued functions by Naud  in \cite[pp. 140-144]{NAUD}. Naud makes some use of taking quotients of  values of functions that we will have to work around.

We give the details now. Recall that $\e', \theta$ are still undetermined. The following technical lemma is the vector valued version of \cite[Lemma 5.10]{NAUD}. Recall that $c_q : I \to U( \C^{\G_q  } )$ is our twisting unitary valued map at level $q$. The function defined on $T^{-N}(I)$ by
\begin{equation}
c_q^N( x ) \equiv c_q( T^{N-1} x ) c_q ( T^{N-2} x ) \ldots c_q(Tx ) c_q( x )  \label{defineourselvesacocyle}
\end{equation}
arises in our analysis. In particular we will need to consider $c_q^N( \a_i^N x)$, where $\a^N_i$, $i=1,2$ are the two  particular  branches of $T^{-N}$ that are given by Lemma \ref{NLIcons}. We record the key fact here that since $c_q$ is locally constant, so too is $c^N_q$ for any $N$. 

\begin{lemma}[Key technical fact towards non-stationary phase]\label{technical}
Let $H \in \CC_{A|b|}$, $f \in C^1(I; \C^{\G_q})$ such that $|f| \leq H$ and $|f'| \leq A|b|H$. For $i = 1 , 2,$ define for $\theta$ a small real parameter and for any $q$,
\begin{equation}
\Theta_1(x) \equiv \frac{ |e^{ [ \tau_a^N + ib \tau^N](\a_1^N x)}c_q^N( \a_1^N x)f(\a_1^N x) +e^{ [ \tau_a^N + ib \tau^N](\a_2^N x)}c_q^N( \a_2^N x)f(\a_2^N x) |}{ (1-2\theta) e^{ \tau_a^N(\a_1^N x) }H(\a_1^N x) +e^{ \tau_a^N(\a_2^N x) }H(\a_2^N x) };
\end{equation}
\begin{equation}
\Theta_2(x) \equiv \frac{ |e^{ [ \tau_a^N + ib \tau^N](\a_1^N x)}c_q^N( \a_1^N x)f(\a_1^N x) +e^{ [ \tau_a^N + ib \tau^N](\a_2^N x)}c_q^N( \a_2^N x)f(\a_2^N x) |}{  e^{ \tau_a^N(\a_1^N x) }H(\a_1^N x) +(1-2\theta)e^{ \tau_a^N(\a_2^N x) }H(\a_2^N x) }.
\end{equation}
Then for $N$ large enough, one can choose $\theta$ and $\e'$ small enough such that for $j$ with $X_j \cap K \neq \emptyset$, 
there are $j'$ with $|j - j'| \leq 2$, $X_{j'} \cap K \neq \emptyset$ and $i \in \{ 1 , 2\}$ such that
\begin{equation}
\Theta_i(x) \leq 1 \quad\text{for all $x \in X_{j'}$.}
\end{equation}

\end{lemma}

Before giving the proof we must state a simple Lemma from \cite{NAUD}. The proof goes through easily in our vector valued setting. This is also covered in \cite[Lemma 3.29]{OW}.

\begin{lemma}[Lemma 5.11 of \cite{NAUD}]\label{alt}
 Let $Z \subset I$ be an interval with $| Z| \leq c / |b|$. Let $H \in \CC_{A|b|}$ and $f \in C^1( I; \C^{\G_q} )$ with $|f| \leq H$ and $|f'| \leq A|b|H$. Then for $c$ small enough, we have either
\begin{equation}
| f(u) | \leq \frac{3}{4}H(u)
\quad\text{for all $u \in Z$, or}
\end{equation}
\begin{equation}
|f(u) | \geq \frac{1}{4} H(u) \quad\text{for all $u \in Z$.}
\end{equation}
\end{lemma}

We also need the following piece of trigonometry from \cite[Lemma 5.12]{NAUD}.
\begin{lemma}[Sharp triangle inequality]\label{trig}
Let $V$ be a finite dimensional complex vector space with Hermitian inner product $\langle \bullet , \bullet \rangle$. For non zero vectors $z_1 , z_2$ with $|z_1 | / |z_2| \leq L$ and
\begin{equation}\label{eq:subconvex}
\Re \langle z_1 , z_2 \rangle \leq (1 - \eta) |z_1| |z_2| ,
\end{equation}
there is $\delta = \delta( L ,\eta)$ such that
\begin{equation}
|z_1 + z_2| \leq (1-\delta)|z_1| + |z_2|.
\end{equation}
\end{lemma}

\begin{proof}[Proof of Lemma \ref{technical}]
Choose $\e'$ small enough so that Lemma \ref{alt} holds for all $Z = Z^i_j$ (with $c = \e'$). As in \cite{NAUD} by choosing $N$ large enough it is possible to assume $|Z^i_j| \leq |V_j|$ for all $j , i$. We also enforce $\theta < 1/8$ so that $1-2\theta \geq 3/4$.

Now let $V_j , V_{j+1}, V_{j+2}$ all have non empty intersection with $K$. One of the $j , j+1 , j+2$ will be the $j'$ of the Lemma. Set $\Xhat_j = X_j \cup X_{j+1} \cup X_{j+2}$ and assume as in Naud that $\Xhat_j$ is contained in one connected component of $I$; note that $\Xhat_j$ is connected.

Following from  our choice of $\theta$, if there is $j' \in \{ j , j+1 , j+2\}$ and $i \in \{1 ,2\}$ with $|f(u)| \leq \frac{3}{4}H(u)$ when $u \in Z^i_{j'}$ then $\Theta_i(u) \leq 1$ on $Z^i_{j'}$ and we are done. So we can assume $ |f(u)| > \frac{3}{4}H(u)$
for some $u$ in each $Z^i_{j'}$. Hence by Lemma \ref{alt}, for all $i , j'$ we have
\begin{equation}\label{eq:assumption}
| f(u) | \geq \frac{1}{4}H(u) > 0 , \quad  \forall u \in Z^i_{j'}.
\end{equation}

 We make the definition
\begin{equation}
z_i(x) \equiv \exp\left( [ \tau_a^N + ib \tau^N](\a_i^N x) \right)c_q^N( \a_i^N x) f(\a_i^N x),  \quad z_i : \Xhat_j \to \C^{\G_q} , \quad  i = 1,2.
\end{equation}
The result follows from Lemma \ref{trig} after establishing bounds on the relative size and angle of $z_1, z_2$ uniformly in appropriate $X_{j'}$.

\textbf{Control of relative size.} Firstly we wish to control the relative size of $z_1, z_2$. This is done by Naud and his estimates go through directly in our case, after making all substitutions of the form
\begin{equation}
\left|  \frac{z_1(x) } {z_2(x) } \right| \to \frac{ |z_1(x) |} {|z_2(x) |}
\end{equation}
and bearing in mind that $c_q^N$ is a unitary valued function.
This caters to our inability to divide non zero vectors. The output of Naud's argument in \cite[pp. 141-142]{NAUD} is that given $j' \in \{j , j+1 , j+2\}$, either $|z_1 (x) | \leq  M |z_2(x) |$ for all $x \in X_{j'}$ or  $|z_2 (x) | \leq  M |z_1(x) |$ for all $x \in X_{j'}$, where 
\begin{equation}
M = 4 \exp( 2 N B_a ) \exp(2 A \e' A_1)
\end{equation}
and 
\begin{equation}
B_a = a \|\tau\|_\infty + |P(-a\tau)| + 2 \| \log h_a \|_\infty
\end{equation}
is a locally bounded function that arises in the estimation of $\tau_a^N$ (cf. \cite[pg. 139]{NAUD}). Returning to the overall argument, this means that we are done when we can establish \eqref{eq:subconvex} with some $\eta$ uniformly on some $X_{j'}$.

\textbf{Control of relative angle.} The key argument here is to very carefully control the angles between the functions $z_1$ and $ z_2$. One sets
\begin{equation}
\Phi(x) \equiv \frac{ \langle  z_1(x) , z_2(x) \rangle }{|z_1(x)||z_2 (x)|} 
\end{equation}
which is the same as
\begin{equation}
\Phi(x) = \exp( i b (\tau^N(\a_1^N x)  - \tau^N(\a_2^N x))) \frac{ \langle c^N_q( \a_1^N x) f(\a_1^N x)  , c_q^N(\a_2^N x )f(\a_2^N x)\rangle  }{|f(\a_1^N x) || f(\a_2^N x)| }.
\end{equation}
Define
\begin{equation}
u_i(x) = c_q^N( \a_i^N x ) \frac{ f(\a_i^N x) }{ | f( \a_i^N x ) | } , \quad x \in \Xhat_j , \quad i = 1,2.
\end{equation}
Then the $u_i$ are $C^1$ as $f$ is non vanishing through \eqref{eq:assumption}. We have
\begin{equation}
(c_q^N  . f  ) \circ \a_i^N =   |f \circ \a_i^N |  . u_i ,
\end{equation}
so that, differentiating on both sides and using $(c_q^N)' \equiv 0$,
\begin{equation}
(c_q^N \circ \a_i^N) .(f \circ \a_i^N )' = | f \circ \a_i^N |' u_i +  |f \circ \a_i^N  | u_i' .
\end{equation}
As $u_i$ has constant length 1 it follows that $u_i$ and $u'_i$ are orthogonal (in $\R^{2|\G_p | }$). Therefore
\begin{equation}
|  [ f \circ \a_i^N ]' |^2 =  ( | f \circ \a_i^N |' )^2 +  | f \circ \a_i^N |^2 | u'_i |^2 .
\end{equation}
It now follows that
\begin{equation}
| u'_i(x) | \leq   \frac{ | [ f \circ \a_i^N ]'(x) | }{ |f ( \a_i^N x ) | } .
\end{equation}
We estimate the right hand side by a direct calculation using the chain rule with the expanding property of $T$ and our assumptions on $H$ from \eqref{eq:assumption} and the hypotheses of Lemma \ref{technical}. Indeed, Naud performs a similar calculation \cite[pg. 142]{NAUD} which yields
\begin{equation}\label{eq:uest}
|u'_i(x) | \leq 8 A |b| \frac{D}{\g^N}.
\end{equation}
Note that we can rewrite the central quantity $\Phi$ as
\begin{equation}\label{eq:phisimple}
\Phi(x)=   \exp( i b (\tau^N(\a_1^N x)  - \tau^N(\a_2^N x)) \langle   u_1(x) ,  u_2(x) \rangle .
\end{equation}
We can use \eqref{eq:uest}  and Cauchy-Schwarz to get
\begin{equation}\label{eq:bracketest}
\left| \frac{d}{dx} \langle u_1 , u_2 \rangle \right|  =\left|  \langle u'_1 , u_2 \rangle +  \langle u_1 , u'_2 \rangle \right| \leq  16 A |b| \frac{D}{\g^N}.
\end{equation}
Note that we have the diameter bound
\begin{equation}\label{eq:diameter}
\diam (\Xhat_{j} ) \leq 3A_1 \frac{ \e' }{|b|} \| (\psi^{-1} )'\|_\infty
\end{equation}
so that using  \eqref{eq:bracketest} 
we have
\begin{equation}
| \langle u_1(x_1) , u_2(x_1) \rangle - \langle u_1(x_2) , u_2(x_2 ) \rangle | \leq 3\cdot 16 \cdot AA_1 \| (\psi^{-1} )'\|_\infty \e' \frac{D}{\g^N}
\end{equation}
for any $x_1, x_2  \in \Xhat_j$; note here that the cocycles $c_q^N(\alpha_i^Nx)$ are constant on $\hat X_j$.
\textbf{We now enforce} $\e' < 1/10$ and $N$ large enough so that
\begin{equation}
48 \cdot AA_1 \| (\psi^{-1} )'\|_\infty \frac{D}{\g^N} < 1.
\end{equation}

Let us cut off one branch of reasoning. Suppose that there is $x_0 \in \Xhat_j$  with
\begin{equation}
   |  \langle u_1(x_0 ) , u_2(x_0 ) \rangle | < 1/10. 
\end{equation}
Then for all $x \in \Xhat_j$ we have
\begin{equation}
 |  \langle u_1(x ) , u_2(x ) \rangle |  < 1/5.
\end{equation}
It would follow that  $ |\Re \Phi(x) | < 1/5 $ for all $x \in \Xhat_j$ and the  Lemma would be proved by our argument with trigonometry. 

Therefore we can now assume 
\begin{equation}
|  \langle u_1(x ) , u_2(x ) \rangle |  \geq 1/10
\end{equation}
for all $x \in \Xhat_j$. Then the new function
\begin{equation}
U(x) = \frac{ \langle u_1(x) , u_2(x) \rangle }{ | \langle u_1(x) , u_2(x) \rangle  | } \in \C
\end{equation}
is $C^1$ on $\Xhat_j$ of constant length $1$ and by an argument we have made before
\begin{equation}\label{eq:derivU}
| U' (x) | \leq \frac{  | \langle u_1 , u_2 \rangle'(x) | }{  | \langle u_1(x) , u_2(x) \rangle  | } \leq 10\cdot 16\cdot  A |b| \frac{D}{\g^N},
\end{equation}
using \eqref{eq:bracketest}. We can write 
\begin{equation}
U(x) = \exp( i \phi(x) ) 
\end{equation}
for some $C^1$ real valued $\phi : \Xhat_j \to \R$.  Then \eqref{eq:derivU} reads
\begin{equation}\label{eq:phideriv}
 |\phi'(x) |  \leq   160 A |b| \frac{D}{\g^N} .
\end{equation}

As we assume $\Phi \neq 0$  on $\Xhat_j$, we can find a $C^1$ function that we will denote
\begin{equation}
\arg \Phi : \Xhat_j \to S^1 = \R /2\pi\Z, \quad \Phi(x) = \exp( i \arg \Phi(x) ) \cdot  |\Phi(x) |.
\end{equation}

Now define
\begin{equation}
F(x)  = (\tau^N(\a_1^N x)  - \tau^N(\a_2^N x))  ,\quad x \in \Xhat_j.
\end{equation}
The critical output of the (NLI) property for $\tau$, Lemma \ref{NLIcons}, tells us  that 
\begin{equation}\label{eq:Fderiv}
0 < m \leq | F'(x) | \leq m' 
\end{equation}
when we choose $N > N_0$, \textbf{which we do}.
As 
\begin{equation}
\arg \Phi =  b F + \phi
\end{equation}
we now have, incorporating \eqref{eq:Fderiv} and \eqref{eq:phideriv}
\begin{equation}
   |b| ( m -        10\cdot16 A  \frac{D}{\g^N} )               \leq |  ( \arg \Phi )' | \leq  |b| ( m' + 10\cdot 16 A \frac{D}{\g^N} ) .
\end{equation}
We \textbf{fix, finally,} $N$ large enough so that we gain $C_2 > C_1 > 0$ (depending only on $N$, $m$, $m'$, $A$, $D$, and $\g$) with 

\begin{equation} 
|b| C_1 \leq | (\arg \Phi )' | \leq |b| C_2 . 
\end{equation}
Now  by estimating diameters of $X_{j+1}$ and $\Xhat_j$ from Proposition \ref{triad} together with the mean value theorem, the total cumulative  change of argument of $\Phi$ between $x_j\in X_j$ and $x_{j+2} \in X_{j+ 2}$, written $\Delta$, is between

\begin{equation}
C_3 \e' \leq \Delta \leq   C_4 \e'
\end{equation}
where
\begin{equation}
C_3 = C_1 A'_1 \inf_{U_0 }| (\psi^{-1})'| > 0 , \quad C_4 =  C_2  3A_1  \| (\psi^{-1} )'\|_\infty .
\end{equation}
\textbf{We now enforce} $\e' < \pi / (2C_4)$ so that we no longer need to worry about $\arg \Phi$ winding around the circle.
 We are about to conclude. Now $\e'$ \textbf{is fixed}. By our trigonometric strategy, we are done with 
\begin{equation}
\theta = \delta\left( M , \left(\frac{C_3 \e'}{100} \right)^2\right)
\end{equation}
 unless there exist $x_j \in X_j$ and $x_{j+2} \in X_{j+2}$ with
\begin{equation}
\Re \Phi( x_k) >  1 -   \left( \frac{C_3 \e'}{100}\right)^2, \quad k = j , j+2 .
\end{equation}
In this case, by the Schwarz inequality we know
\begin{equation}
| \Phi(x_k) | \leq 1 \quad k = j , j +2 
\end{equation}
so it follows that now using the principal branch for $\arg$ and e.g. $| \sin x | \leq 2| x| $
\begin{equation}
 |\arg \Phi (x_k ) | \leq C_3 \e' /50 ,\quad k = j , j+2.
\end{equation}
Given that the argument of $\Phi$ moves at least by $C_3 \e'$ in one direction between $x_j$ and $x_{j+2}$ and does not move more than $\pi/2$ (hence does not wind), this is a contradiction.

\end{proof}

We can now conclude this section with
\begin{proof}[Proof of Lemma \ref{core}, property \ref{corethird}]

Choose  $N$, $\theta$ and $\e'$ so that Proposition \ref{blackbox} holds as well as Lemma \ref{technical}. Increasing $N$ if necessary we may also assume that $\frac{D}{\gamma^N} \leq \frac{1}{4}$. 

Suppose we are given $H \in \CC_{A|b| }$ and $f \in C^1( I ; \C^{\G_q} )$ such that $|f| \leq H$ and $|f'| \leq A|b|H$. The second inequality stated in property \ref{corethird} is softer so we prove this first. The complex scalar version of this inequality is proved in \cite[pg. 138]{NAUD}. 

We calculate
\begin{equation}
[L^N_{s,q}f]  (x)= \sum_{\a^N} \exp( [ \tau^N_a + i b \tau^N ](\a^N x)  ) c_q^N(\a^N x)  f(\a^N x).
\end{equation}
where
\begin{equation}
c_q^N(y) = c_q(T^{N-1} y) \ldots c_q(T y) . c_q(y) 
\end{equation}
and the sum is over branches of $T^{-N}$.
Therefore
\begin{align*}
[ L^N_{s,q}f]'(x)   &=    \sum_{\a^N}([ \tau^N_a + i b \tau^N ] \circ \a^N )'(x) \exp( [ \tau^N_a + i b \tau^N ](\a^N x)  ) c_q^N(\a^N x)  f(\a^N x) \\
&+  \sum_{\a^N}  \exp( [ \tau^N_a + i b \tau^N ](\a^N x)  ) c_q^N(\a^N x) (f \circ \a^N )' ( x) ,
\end{align*}
$c_q^N$ being locally constant. Using that $c_q^N$ is unitary and bounding derivatives of $\a^N$ with the eventually expanding property and chain rule gives
\begin{align*}
| [L^N_{s,q}f]'(x) | \leq  \sum_{\a^N} |([ \tau^N_a + i b \tau^N ] \circ \a^N )'(x)| \exp( [ \tau^N_a ](\a^N x)  )   H(\a^N x)  \\
+ \frac{D}{\g^N}  \sum_{\a^N}   \exp( [ \tau^N_a  ](\a^N x)  )  A|b| H(\a^N x) .
\end{align*}
Using the inequality in Proposition \ref{blackbox} and our choice of $N$ we get 
\begin{equation}
| [L^N_{s,q}f]'(x) | \leq \frac{1}{2} A |b| [L_a^N H](x) \leq A |b|  [\N^J_s H] (x)
\end{equation}
given the very mild assumption $\theta < 1/2$.

 Now we turn to the more difficult first inequality of Lemma \ref{core}, property \ref{corethird}. Given that we have established Lemma \ref{technical} in the vector valued setting, the proof follows by the same argument as  in \cite[pg. 143]{NAUD}.  We give the details here for completeness. 

Let $J$ be the set of indices $(i,j)$ where $\Theta_i(x) \leq 1$ when $x \in X_j$. The statement of Lemma \ref{technical} is precisely that this set of indices is dense (recall Definition \ref{dense}) and hence $J \in \E_s$ as required. We will prove
\begin{equation}
| L^N_{s,q} f | \leq \N_s^J H = L_a( \chi_J H ).
\end{equation}
Fix $x$. Notice that if $x \notin \Int X_j$ for any $j$ then  for all branches $\a^N $ of $T^{-N}$, $\a^N x \notin Z^i_j$ and so $\chi_J(\a^N x) = 1$ for any $J$. More generally if $x \notin \Int X_j$ for any $j$ appearing as a coordinate in $J$ then $\chi_J(\a^N x) = 1$.
 Therefore
\begin{equation}
| [L^N_{s , q} f ](x)| \leq \sum_{ \a^N } \exp( \tau_a^N (\a^N x) ) H(\a^N x)  = \N^J_s[H](x).
\end{equation}

We are left to consider $x , J$ such that $x \in \Int(X_j) $ and $J$ contains $(i, j)$ for some $i$. 

Suppose that $(i , j) = (1 , j)$ and $(2 , j) \notin J$. Then for $\a^N \neq \a_1^N$ a branch of $T^{-N}$, $\chi_J(\a^N x) = 1$ (the only other possibility would have been $\a^N = \a_2^N$). Then using $\Theta_1(x) \leq 1$ gives
\begin{align} \label{eq:previouscalc}
| L^N_{s,q}[f](x)| &\leq \sum_{\a^N \neq \a^N_1 , \a^N_2 } \exp(\tau_a^N(\a^N(x) )H(\a^N(x)) \nonumber
\\ &+ (1 - 2\theta)  \exp(\tau_a^N(\a_1^N(x) )H(\a_1^N(x)) +  \exp(\tau_a^N(\a_2^N(x) )H(\a_2^N(x))\nonumber\\ 
&\leq \N^J_s[H](x).
\end{align}
The case $(i, j ) = (2,j)$ and $(1,j) \notin J$ is treated the same way. Finally, if $(1,j)$ and $(2,j)$ are in $J$ then $\Theta_1(x) , \Theta_2(x) \leq 1$ from which one can estimate
\begin{align*}
&| \exp ( [\tau_a^N + ib \tau^N ](\a_1^N x) ) f(\a_1^N x) +  \exp ( [\tau_a^N + ib \tau^N ](\a_2^N x) ) f(\a_2^N x) | \\
&\leq 
 (1 - \theta)  \exp(\tau_a^N(\a_1^N(x) )H(\a_1^N(x)) + (1-\theta)  \exp(\tau_a^N(\a_2^N(x) )H(\a_2^N(x)) \\
&\leq  \exp(\tau_a^N(\a_1^N(x) )\chi_J(\a_1^N x) H(\a_1^N(x)) +  \exp(\tau_a^N(\a_2^N(x) )\chi_J(\a_2^N x) H(\a_2^N(x)).
\end{align*}
Also noting that $\chi_J(\a^N x) =1$ when $\a^N \neq \a_i^N$, $i = 1,2$,  the previous inequality shows
\begin{equation}
| L^N_{s,q}[f](x)| \leq \N^J_s[H](x)
\end{equation}
in our final remaining case. The proof is complete.
\end{proof}

\section{Bounds for transfer operators: small imaginary part}\label{smallimaginarypart}

In this section we aim to supplement  Lemma \ref{powers} with uniform bounds for powers of transfer operators for small imaginary part of $s$. In this regime the bounds are due to a mixing property of the groups $\G_q$ that we explain now.

\subsection{Ingredients}\label{expingredients}
The critical ingredient is strong spectral radius estimates for complex valued measures on $\G_q$. We  define $E_q$ to be the space of functions of $\G_q$ that are orthogonal to all functions lifted from $\G_{q'}$ for $q' | q$.

Our strong spectral estimates are furnished by
\begin{dfn}[Property (MIX)]\label{KSdfn}
Let $\G$ be a free semigroup  spanned by generators $S \subset \SL_2(\Z)$ and consider the associated measure
\begin{equation}
\nu = \sum_ { s \in S } \delta_s 
\end{equation}
on $\SL_2(\Z)$. We write $\nu_q$ for the projection of this measure modulo $q$ and $\nu_q^{(R)}$ for the $R$ fold convolution. We say that $\G$ has the mixing property (MIX)  on the index set  $\Q \subset \mathbf{N}$ if for any $c_1 > 0$, there exist positive 
$c_2, c_3, c_4$ and $q_0$  such that for any $q \in \Q$ with $q > q_0$,  any  complex valued measure $\mu$ on $\G_q$  satisfying
\begin{equation}
\| \mu \|_1 < B 
\end{equation}
and
\begin{equation}
| \mu(x) | \leq Bq^{-c_1c_2} \nu_q^{(\lceil c_2 \log q \rceil)} (x) , \quad x \in \G_q
\end{equation}
 also  satisfies
\begin{equation}
\| \mu \star \varphi \|_2 \leq c_4 B q^{-c_3}  \| \varphi \|_2\quad\text{ for $\varphi \in E_q$.}
\end{equation}

\end{dfn}
In the paper \cite{BGSACTA}, Bourgain, Gamburd and Sarnak have established a flattening property which implies the analog of (MIX) for convex cocompact\footnote{Hence free by \cite{BUTTON}.} sub\textit{groups} of $\SL_2(\Z)$ and square free $\Q$ avoiding some finite modulus. This implication is also given in \cite{BGSACTA}. In Section \ref{modular}, we show how this implication can be made to work in our setting for square free index sets $\Q$.

We can now state the main Lemma of this Section.
\begin{lemma}\label{powerssmallfreq}
Let $q \in \Q$ where the semigroup $\G$ has the property (MIX).  For $b_0 > 0$ given, there are $a_0$, $q_0$, $\k$ and $\delta$ such that when $|a - s_0| < a_0$, $|b| < b_0$ and $N = \lceil  \k \log q \rceil$ with $q > q_0$ we have
\begin{equation}
 \| L^{n N }_{s,q} W \|_\Lip   \leq q^{- n \delta }  
\end{equation}
for all $E_q$-valued $W \in C^1( I ; \C^{\G_q} )$ with $\|W\|_\Lip = 1$. 
\end{lemma}
The proof of this Lemma is given in the remainder of this Section.
\subsection{Relevance of mixing estimates}
We are now going to show how  mixing estimates arise naturally in the consideration of $L^N_{s,q}$. We have calculated already that for $W \in C^1(I)$ with $\| W \|_{C^1 } < \infty$ and taking on values only in the orthocomplement to constant functions 
\begin{equation}\label{eq:LNWexplicit}
[L^N_{s,q}W]  (x)= \sum_{\a^N} \exp( [ \tau^N_a + i b \tau^N ](\a^N x)  ) c_q^N(\a^N x)  W(\a^N x)
\end{equation}
where the sum is over branches of $T^{-N}$. It will be convenient to make the parametrization
\begin{equation}
N = M + R , \quad M , R > 0.
\end{equation}
For given $\a^N$ we can write uniquely
\begin{equation}
\a^N = \a^{M} \a^{R}
\end{equation}
where $\a^{M}$, $\alpha^{R}$ are branches of $T^{-M}$ and $T^{-R}$. In this case we write $\a^N > \a^M$. For each $i \in \lbrace 1 \ldots j \rbrace$ we choose $x_j \in I_J \cap K$.  We notice here that  for any 
 $x\in I$ with $\alpha^R(x) \in I_j$, if $\a^N > \a^M$ we have
\begin{equation}
d( \a^{N} x , \a^M x_j ) =  d(  \a^{M}  (\a^{R} x) ,  \a^{M} x_j ) \leq \frac{D}{\g^M} \diam( I )
\end{equation}
by the expanding property of $T$. Then 
\begin{equation}
\| W( \a^N x)  - W(\a^M x_j ) \| \leq   \frac{D}{ \g^M } \diam ( I )  \|W \|_\Lip.
\end{equation}
It follows then that 
\begin{align}
[L^N_{s,q}W]  (x) &=  \sum_{\a^N} \exp( [ \tau^N_a + i b \tau^N ](\a^N x)  ) c_q^N(\a^N x)  W(\a^M x_j) \\
&+ O \left( \| W \|_\Lip  \frac{D}{ \g^M } \diam ( I )  \sum_{ \a^N } \exp ( \tau^N_a (\a^N x)  )   \right).
\end{align}
We will assume that $D \g^{-M}$ is small, say $< 1/( 100 \diam(I))$ and note that the sum in the error term is
\begin{equation}
  \sum_{ \a^N }  \exp ( \tau^N_a (\a^N x)  )   = L^N_a [ 1 ](x)  = 1(x) =1 
\end{equation}
as the operator has been normalized. So 
\begin{equation}\label{eq:uncoupled}
[L^N_{s,q}W]  (x) =  \sum_{\a^N} \exp( [ \tau^N_a + i b \tau^N ](\a^N x)  ) c_q^N(\a^N x)  W(\a^M x_j)  + O ( \|W \|_\Lip \g^{-M } );
\end{equation}
we are abusing notation slightly here, inasmuch as the choice $x_j$ depends on $M$. This is an important estimate as it allows us access the expansion properties coming from $c_q$ by decoupling $M$ and $N$. 

Recall that $c_q$ was obtained by reducing $c_0 \mod q$ to obtain a mapping $c_q : I \to \G_q$. This mapping was reinterpreted as $c_q : I \to U ( \C^{\G_q} )$ via the right regular representation of $\G_q$.

For any specified $\a^M$ and $x \in I$ we construct the complex valued measure on $\G_q$
\begin{equation}
\mu_{s , x , \a^M } = \sum_{ \a^N > \a^M }  \exp( [ \tau^N_a + i b \tau^N ](\a^N x)  ) \delta_{c_q^R ( \a^R x )}
\end{equation}
where $\delta_{g}$ gives mass one to $g \in \G_q$.

For any $f \in C^1( I ; \C^{\G_q} )$,  $x \in I$ and $\a^M$ we construct complex valued measure $\varphi_{f,x,\a^M}$ by
\begin{equation}
\varphi_{f,\a^M} ( g ) = \sum_{ g \in \G_q } f ( \a^M x_j )\lvert_{g}  \:  \delta_{ c^M_q( \a^M x_j ) g  }
\end{equation}
where $c_q$ is thought of as $\G_q$ valued and $f(\a^M x_j)$ thought of as a $\C$-valued function on $\G_q$. The key fact here is that
\begin{align}
 [ \mu_{s , x , \a^M } \star \varphi_{f,\a^M} ] &= \sum_{g \in \G_q } \sum_{ \a^N > \a^M }  \exp( [ \tau^N_a + i b \tau^N ](\a^N x)  ) f ( \a^M x_j )\lvert_g   \delta_{c_q^R ( \a^R x )} \star   \delta_{ c^M_q( \a^M x_j ) g  } \\
&=  \sum_{g \in \G_q } \sum_{ \a^N > \a^M }  \exp( [ \tau^N_a + i b \tau^N ](\a^N x)  ) f ( \a^M x_j )\lvert_g   \delta_{c_q^N ( \a^N x )g} .
\end{align}
This means that
\begin{equation}
 [ \mu_{s , x , \a^M } \star \varphi_{f,\a^M} ]  =  \sum_{ \a^N > \a^M }  \exp( [ \tau^N_a + i b \tau^N ](\a^N x)  ) c_q^N(\a^N x) f ( \a^M x_j ).
\end{equation}
The reader should compare this with \eqref{eq:uncoupled}.

\subsection{Bounds for $\mu_{s , x, \a^M }$ }

We need bounds for $\|  \mu_{s , x , \a^M } \|_{1}$ and $ | \mu_{s , x , \a^M }(y)| $ pointwise in order to use the mixing feature of the $\G_q$. Let us bound the $L^1$ norm first. Firstly we write
\begin{equation}\label{eq:pointwisebound}
| \mu_{s , x , \a^M } | \leq \sum_{ \a^N > \a^M }  \exp( \tau^N_a (\a^N x)  ) \delta_{c_q^R ( \a^R x )}.
\end{equation}
Notice that
\begin{equation}\label{eq:splitup}
\tau_a^N ( \a^N x ) = \tau_a^M(\a^N x) + \tau_a^R(\a^R x)  .
\end{equation}
Then
\begin{equation}
 \| \mu_{s , x , \a^M } \|_1 \leq   \sum_{ \a^N > \a^M }  \exp( \tau^M_a (\a^N x)  ) \exp( \tau^R_a( \a^R x) ) .
\end{equation}
We now decouple: let $\a^N_0 = \a^N_0(\a^M)$ be some choice of $\a^N_0 > \a^M$. Then
\begin{equation}
 \tau^M_a (\a^N x)    -  \tau^M_a (\a^N_0 x)  = \sum_{ n = 0}^{M-1} \tau_a ( T^n \a^N x ) - \tau_a( T^n \a^N_0 x )
\end{equation}
and noting that $T^n \a^N x $ and $ T^n \a^N_0 x$ are within 
\begin{equation}
\frac{D}{\g^{ M - n } } \diam(I) 
\end{equation}
of one another, we have
\begin{align}
\tau_a^M(\a^N x) &\leq \tau_a^M(\a_0^N x) + D . \diam(I) \sup_{ y \in I } | [\tau_a]'(y) | \sum_{ n =  0}^{M-1} \frac{1}{\g^{ M - n }} \\
& \leq \tau_a^M(\a_0^N x) + \k_1 ( D, \g , I , \tau , a_0) 
\end{align}
for $| a- s_0 | < a_0$ (as $\tau_a$ is roughly constant in $a$ close to $s_0$). Therefore
\begin{align}
\| \mu_{s , x , \a^M } \|_1 &\leq  \exp( \k_1  + \tau^M_a (\a^N_0 x)  )  \sum_{ \a^N > \a^M }\exp( \tau^R_a( \a^R x) )  \\
&\leq  \exp( \k_1  + \tau^M_a (\a^N_0 x)  )  [ L^R_a 1 ](x) =  \exp( \k_1  + \tau^M_a (\a^N_0 x)  )  .
\end{align}
by the normalization of $L_a$. We record this bound in the following.
\begin{lemma}\label{mul1}
Given $a_0$ small enough,  $x$ and $\a^M$, there is $\k_1  =\k_1(  a_0)$ such that
\begin{equation}
\| \mu_{s , x , \a^M } \|_1 \leq \exp( \k_1  + \tau^M_a (\a^N_0 x)  ) ,
\end{equation}
for $|a- s_0| < a_0$. Here $\a^N_0$ is any branch of $T^{-N}$ such that $\a^N_0 > \a^M$.
\end{lemma}
Now we turn to the pointwise bound. We have immediately from  \eqref{eq:splitup} that
\begin{equation}
|  \mu_{s , x , \a^M } (y) |  \leq \sup_{ \a^N > \a^M }   \exp(  \tau^M_a ( \a^N x)  ) \exp ( \tau^R_a(\a^R x) )    \nu_q^{(R)} (y) 
\end{equation}
where $\nu_q^{(R)}$ is the measure associated to the random walk of length $R$ $\bmod \:q$ in the generators $S$.

Now repeating the arguments leading to Lemma \ref{mul1} gives
\begin{equation}
|  \mu_{s , x , \a^M } (y) | \leq   \exp( \k_1  + \tau^M_a (\a^N_0 x)  )  \exp ( \tau^R_a(\a^R x) ) \nu_q^{(R)} (y)  .
\end{equation}
As in \cite[pg. 133]{NAUD}, for any $\eta > 0$ there is $a_0>0$ such that when $|a - s_0| < a_0$ 
\begin{equation}
| \tau_a^R(x) - \tau_{s_0}^R (x) | \leq \eta R .
\end{equation}
Therefore for any $\eta > 0$, by forcing $|a - s_0 | < a_0$ we can ensure
\begin{equation}
|  \mu_{s , x , \a^M } (y) | \leq   \exp( \k_1  + \tau^M_a (\a^N_0 x)  )  \exp ( \eta R +  \tau^R_{s_0}(\a^R x) ) \nu_q^{(R)} (y)  .
\end{equation}
One calculates
\begin{equation}
|  \tau^R_{s_0}(\a^R x) )  + s_0 \tau^R( \a^R x ) | = | \log h_{s_0}(  x) -\log  h_{s_0} ( \a^R x ) | \leq \k_3
\end{equation}
as $h_{s_0}$, the eigenfunction of $\L_{-s_0\tau}$ with eigenvalue $1$, is bounded above and below away from zero. 
By the eventually positive property of $\tau$, there exist $N_0$ and $\k_4 > 0$ such that $\tau^{N_0} \geq \k_4$ on $T^{-N_0}(I)$. Together with the fact that $\tau$ is bounded this implies there are $\k_5$ and $\k_6$ such that
\begin{equation}
\tau^R( \a^R x ) \geq \k_6 R
\end{equation}
when $R > \k_5$.
Putting our bounds together, by choosing $\eta$ much less than $\k_6$ we have proved
\begin{lemma}\label{pointwisemu}
There are $a_0 , \k_1 , \k_2 , \k_5 $ and $\k_7 > 0$ such that when $|a - s_0 | < a_0$ and $\k_5 < R  $ we have
\begin{equation}
|  \mu_{s , x , \a^M } (y) | \leq  \exp( \k_1  + \tau^M_a (\a^N_0 x)  ) \exp( - \k_7 R ) \nu_q^{(R)} (y) 
\end{equation}
for $y \in \G_q$.
Here $\a_0^N$ can be any branch of $T^{-N}$ such that $\a_0^N > \a^M$. 
\end{lemma}

We can now appeal to Definition \ref{KSdfn} (assuming the property (MIX) holds) to get by using  \eqref{eq:uncoupled},  Lemma \ref{mul1} and Lemma \ref{pointwisemu} that

\begin{align*}
\| [ L^N_{s,q} W ] (x) \| &\leq  \sum_{\a_M }  \| [ \mu_{s , x , \a^M } \star \varphi_{W,\a^M} ] \|_{l^2(\G_q)} + O( \|W \|_\Lip \gamma^{-M} ) \\
&\leq  c_4 q^{ - c_3} \exp( \k_1 ) \sum_{\a^M } \exp ( \tau^M_a (\a^N_0 x)  ) \|  \varphi_{W,\a^M} \|_{l^2(\G_q) }  + O( \| W \|_\Lip \gamma^{-M} )
\end{align*}
for some $c_2 , c_3, c_4$ provided by Definition \ref{KSdfn} when setting $c_1 = \k_7$ by taking 
\begin{equation}
R = \lceil c_2 \log q \rceil. 
\end{equation}
We have chosen for each $\a^M$ an $\a^{N}_0$ with $\a_0^N > \a_M$ and we are assuming the conditions in the Lemmas we have used are met. Since trivially 
\begin{equation}
 \|  \varphi_{W,\a^M} \|_{l^2(\G_q) }  \leq \| W \|_\infty
\end{equation}
we can continue to bound $\| [ L^N_{s,q} W ] (x) \| $ up to $O( \|W\|_\Lip \g^{-M})$ by
\begin{align*}
 c_4 q^{ - c_3 } \exp( \k_1 ) \| W \|_\infty \sum_{\a^M } \exp ( \tau^M_a (\a^N_0 x)  ) &\leq  c_4q^{ - c_3 } \exp( \k_1 ) \| W \|_\infty L^N_a[1](T^M\a^N_0 x)\\
 &=  c_4 q^{ - c_3  } \exp( \k_1 ) \| W \|_\infty.
\end{align*}
We have now proved, by choosing $N > \k_{10}\log q$ so that there is room for the requisite $R$ and big enough $M$ the following Lemma.
\begin{lemma}\label{infinitybound}
If the semigroup $\G$ has property (MIX)  for $q \in \Q$, there are $a_0 , q_0 , \k_{10}, \delta > 0$ and $\g' > 1$ such that when $|a - s_0 | < a_0$,  we have 
\begin{equation}
\| L^N_{s,q} W \|_\infty \leq q^{ -\delta} \| W \|_\infty +{\g '}^{-N }  \| W \|_\Lip
\end{equation}
when $N >\k_{10} \log q$, $q > q_0$, $q \in \Q$ and $W \in E_q$ with $\|W\|_\Lip < \infty$.
\end{lemma}

\subsection{Bounds for Lipschitz norms}
 In order to iterate Lemma \ref{infinitybound} we also need bounds for
\begin{equation}
\| L^N_{s,q} W \|_\Lip
\end{equation}
under the same conditions as in Lemma \ref{infinitybound}. This amounts to estimating
\begin{equation}
\sup_I | [ L^N_{s,q} W ]' |
\end{equation}
and so we can proceed along similar lines as before. Indeed one calculates from \eqref{eq:LNWexplicit} that
\begin{align}
[L^N_{s,q} W ]'(x) &= \sum_{\a^N }( [\tau^N_a + ib\tau^N]\circ \a^N )'(x)   \exp( [ \tau^N_a + i b \tau^N ](\a^N x)  ) c_q^N(\a^N x)  W(\a^N x) \\
&+ \sum_{\a^N}\exp( [ \tau^N_a + i b \tau^N ](\a^N x)  ) c_q^N(\a^N x) [W \circ \a^N]'(x)
\end{align}
using that $c_q^N$ is locally constant. The second set of terms are bounded by
\begin{equation}
\frac{D}{\g^N }    \sum_{\a^N } \exp(  \tau^N_a(\a^N x)   )   \|W\|_\Lip   
\end{equation}
which can be  bounded by 
\begin{equation}
\frac{D}{\g^N } \| W \|_\Lip L_a^N[1](x) = \frac{D}{\g^N } \| W \|_\Lip.
\end{equation}
So we have
\begin{equation}\label{eq:Ssplit}
[L^N_{s,q} W ]'(x) = \Sigma +  O ( \frac{D}{\g^N } \| W \|_\Lip )
\end{equation}
where 
\begin{equation}
\Sigma \equiv \sum_{\a^N }( [\tau^N_a + ib\tau^N]\circ \a^N )'(x)   \exp( [ \tau^N_a + i b \tau^N ](\a^N x)  ) c_q^N(\a^N x)  W(\a^N x) .
\end{equation}
We can go through the same decoupling argument as before to get
\begin{align}
\Sigma&=  \sum_{\a^N}( [\tau^N_a + ib\tau^N]\circ \a^N )'(x) \exp( [ \tau^N_a + i b \tau^N ](\a^N x)  ) c_q^N(\a^N x)  W(\a^M x_j) \\
&+ O \left( \| W \|_\Lip  \frac{D}{ \g^M } \diam ( I )  \sum_{ \a^N } | ([\tau^N_a + ib\tau^N]\circ \a^N )'(x)| \exp ( \tau^N_a (\a^N x)  )   \right). 
\end{align}
Note that since there are  constants $C$ and $a_0$ such that when $| a - s_0| < a_0$ we have
\begin{equation}
| [\tau_a^N \circ \a^N ]'(x) | \leq C
\end{equation}
for $x \in I$ (see for example \cite[pg. 138]{NAUD}), we have
\begin{equation}\label{eq:k11}
|  [ [ \tau_a^N + i b \tau^N ] \circ \a^N ]' (x) | \leq C + | b | \sup_{I} | \tau'|  \sum_{ i = 0 }^{N-1} \frac{D}{\g^N } \leq \k_{11}
\end{equation}
for some $\k_{11} = \k_{11}( a_0 , b_0 )$ when $|b| \leq b_0$. Therefore we have the decoupled equation
\begin{equation}
\Sigma =  \sum_{\a^N}( [\tau^N_a + ib\tau^N]\circ \a^N )'(x) \exp( [ \tau^N_a + i b \tau^N ](\a^N x)  ) c_q^N(\a^N x)  W(\a^M x)  + O_{b_0}(\| W \|_{\Lip} \g^{-M} )
\end{equation}
valid when $|b| < b_0$ and $|a| < a_0$ for some fixed $a_0$. We denote the first of these two terms by $\Sigma'$. Now similarly to before we define complex valued measures
\begin{equation}\label{eq:mudashdef}
\mu'_{s , x , \a^M } = \sum_{ \a^N > \a^M }  ( [\tau^N_a + ib\tau^N]\circ \a^N )'(x)\exp( [ \tau^N_a + i b \tau^N ](\a^N x)  ) \delta_{c_q^R ( \a^R x )}
\end{equation}
\begin{equation}
\varphi_{f,\a^M} ( g ) = \sum_{ g \in \G_q } f ( \a^M x_j )\lvert_{g}  \:  \delta_{ c^M_q( \a^M x_j ) g  }
\end{equation}
for $f \in C^1(I ; \G^q)$, $\a^M$ a branch of $T^{-M}$. Then the key observation is that
\begin{equation}\label{eq:keyS}
 \| \Sigma '\|  = \left\| \sum_{\a^M } \mu'_{s , x , \a^M } \star \varphi_{W,\a^M} \right\|_{l^2( \G_q ) } .
\end{equation}

\subsection{Bounds for $\mu'_{s , x , \a^M }$}
We have 
\begin{equation}
\| \mu'_{s, x ,\a^M } \|_1 \leq \sup_ I | [\tau^N_a + ib\tau^N]\circ \a^N )'(x) |  \sum_{\a^N > \a^M } \exp(  \tau^N_a (\a^N x)  ) .
\end{equation}
We bounded the sum before as $\exp ( \k_1 + \tau_a^M(\a_0^N x) )$ where $\a_0^N$ is any branch of $T^{-N}$ such that $\a_0^N > \a^M$ and $|a  -s_0| < a_0$ for small enough $a_0$. The supremum in front of the sum is bounded by $\k_{11}(b_0)$ when $|b| < b_0$, by \eqref{eq:k11}. Therefore
\begin{lemma}\label{mudashl1}
We can find $a_0 > 0$ such that for $b_0>0$ given, there is $\k_{12} = \k_{12}(a_0 , b_0) >0$ such that
\begin{equation}
\| \mu'_{s,x,\a^M } \|_1 \leq \k_{12} \exp( \tau_a^M ( a_0^N x) )
\end{equation}
whenever $|b| < b_0$ and $|a - s_0 | <a_0$, for each  $x \in T^{-N }I$ and $\a^M$. Here $\a_0^N$ can be any branch of $T^{-N}$ such that $\a_0^N > \a_M$.
\end{lemma}
To get a bound for $| \mu'_{s,x,\a^M }(y)|$ pointwise  we repeat the arguments leading up to Lemma \ref{pointwisemu} to get under the same conditions that
\begin{equation}
\exp(  \tau^N_a  (\a^N x)) \leq \exp( \k_1 + \tau_a^M(\a_0^N x) ) \exp( - \k_7 R ).
\end{equation}
Therefore by incorporating \eqref{eq:k11} we get
\begin{lemma}\label{pointwisemudash}
For any $b_0>0$, one can choose constants such that the following holds. There exist $a_0 > 0$ $, \delta >0 $ so that if $|a - s_0 | < a_0$,  $|b| < b_0$ and $\k_5 < R$ (cf. Lemma \ref{pointwisemu}) then  for sufficiently large $q$
\begin{equation}
| \mu'_{s,x,\a^M }(y)  | \leq \k_{12} \exp( \tau_a^M (\a_0^N x)  )\exp( - \k_7 R ) \nu_q^{(R)}(y)  ,\quad y \in \G_q
\end{equation}
where $\k_{12} = \k_{12}(b_0)$ and $\a_0^N$ is any branch of $T^{-N}$ with $\a_0^N > \a^M$.
\end{lemma}
We now take $R$ as in the preamble to Lemma \ref{infinitybound}, $R = \lceil c_2 \log q \rceil$ where the constants are as before relative to $c_1 = \k_7$ in Definition \ref{KSdfn}.  Recall the constants $c_3 , c_4 , q_0$ from before and assume the mixing  property holds for our range of $q$. For each $\a^M$ choose an $\a_0^N > \a_M$. The arguments leading up to Lemma \ref{infinitybound} apply in our current setting to allow us to estimate $\| \Sigma' \|$.
Indeed, as we have $\| \varphi_{W , x , \a^M } \|_{l^2(\G_q )} \leq \|W \|_\infty$ we can estimate $\|\Sigma'\|$ from \eqref{eq:keyS}, Lemma \ref{mudashl1}, Lemma \ref{pointwisemudash} and Definition \ref{KSdfn} to get for $W \in E_q$
\begin{align}
\| \Sigma' \| &\leq  c_4 \k_{12} \| W \|_\infty q^{-c_3}  \sum_{ \a^M } \exp( \tau_a^M ( \a_0^N x)  )
\\ &\leq    c_4\k_{12} \| W \|_\infty q^{-c_3} L^N_a[1](T^M \a^M_0 x) \leq  C_4 \k_{12} \| W \|_\infty q^{-c_3}
\end{align}
whenever  $|a - s_0| < a_0$, $|b| < b_0$ are the ranges specified by previous Lemmas and $N > \k_{14} \log q$.
It now follows from \eqref{eq:Ssplit} that with these conditions on $N,  q , a , b$ we have in light of Lemma \ref{infinitybound}
\begin{equation}
\| L^N_{s,q} W \|_{\Lip} \leq \k_{15}  q^{-\delta} \| W \|_\infty + \k_{16} \g^{-N} \| W \|_\Lip  + \g '^{-N} \| W \|_\Lip
\end{equation}
for some $\delta > 0$ when $W \in E_q$ and $q \in \Q$ with $q > q_0$. In other words

\begin{lemma}\label{uniterated}
Suppose the (MIX) property holds for the $\G$ relative to $q \in \Q$. Then for any given $b_0$, there are
$a_0 ,  q_0 , \k_{14}, \delta' > 0$  such that when $|a - s_0 | < a_0$  and $|b| < b_0$ we have 
\begin{equation}
\| L^N_{s,q} W \|_\Lip \leq q^{ -\delta'} \| W \|_\Lip
\end{equation}
when $N >\k_{10} \log q$, $q > q_0$,  $q \in \Q$ and  $W \in C^1( I ; \C^{\G_q} )$ with values in $E_q$.
\end{lemma}
Lemma \ref{powerssmallfreq} follows by iterating Lemma \ref{uniterated} and relabeling constants. This concludes our discussion of spectral estimates for small imaginary values.

\subsection{The new subspace structure and the proof of main Theorem \ref{mainthm}}
We can now show how the main Theorem \ref{mainthm}  follows from our estimates. Part \ref{mainsecond} was established in Proposition \ref{bigimaginaryprop} so it remains to show how part \ref{mainfirst} follows from our results.

We note first the following consequence of Lemma \ref{powerssmallfreq}.
\begin{lemma}\label{powerssmallfreq2}
Suppose that the (MIX) property for $\G$ holds for $q \in \Q$.
For all $b_0\in \R$, there are $0 < \rho < 1$, $a_0$, $q_0$ and $C$ such that when $|a -s_0 |<a_0$, $|b|\leq b_0$ and
 $q_0 < q \in \Q$, we have for all $m > 0$
\begin{equation}
\| L^{m}_{s,q} f \|_{C^1} \leq  C q^C \rho^m \| f \|_{C^1} \quad\text{when $f \in E_q$.}
\end{equation}

\end{lemma}
This is an easy exercise and the reader can get the details from the proof of \cite[Theorem 4.3]{OW}.

Recall the new subspace structure of $\G_q$. For any $q' | q$ there is a projection $\G_q \to \G_{q'}$. The kernel of this projection will be denoted $\G_q(q')$, the congruence subgroup of level $q'$ in $\G_q$. These have the property that if $q'' | q'$ then $\G_q(q') \leq \G_q(q'')$. This groups give an orthogonal decomposition of the right regular representation 
\begin{equation}\label{eq:decompregular}
\C^{\G_q} = \bigoplus_{q' | q } E^q_{q'} 
\end{equation}
where $E^q_{q'}$ consists of functions invariant under $\G_q(q')$ but not invariant under $\G_q(q'')$ for any $q'' | q'$, $q'' \neq q'$. Then the $E_q$ from before matches $E^q_q$ as defined here.

The decomposition \eqref{eq:decompregular} gives rise to a corresponding direct sum decomposition
\begin{equation}
C^1( I ; \C^{\G_q})  = C^1(I ) \oplus \bigoplus_{1 \neq q' | q} C^1( I ; E^q_{q'} ).
\end{equation}
It is clear that the subspaces $E^q_{q'}$ are invariant under the transfer operator $L_{s,q }$ and taking derivatives.

Also note that if $f \in E^q_{q'}$ then $f$ descends to a well defined function $F$ on $\G_q / \G_q(q') \cong \G_{q'}$ which is not invariant under any congruence subgroup of $\G_{q'}$, hence in $E^{q'}_{q'}$. Also, if $G$ is a function in $E^{q'}_{q'}$ then $G$ lifts through the previous isomorphism to a function $g$ in $E^q_{q'}$ for any $q' | q$. This gives rise to an map of Banach spaces
\begin{equation}
\Phi_{q,q'} :C^1(I ; E^{q'}_{q'}) \to C^1(I ;  E^{q}_{q'} )
\end{equation}
for any $q' | q$ with the property that
\begin{equation}
\| \Phi_{q,q'} ( f ) \|_{C^1}  = \sqrt{ | \G_q(q') | } \| f \|_{C^1} .
\end{equation}

 This map is equivariant under the transfer operators in the sense that
\begin{equation}
\Phi_{q,q'} [L_{s,q' } f] = L_{s,q} \Phi_{q,q'}[f]
\end{equation}
for any $f \in E^{q'}_{q'}$. In other words, the action of $L_{s,q}$ on a summand in \eqref{eq:decompregular} is determined by the action of the corresponding transfer operator on $E^{q'}_{q'}$ for some $q' | q$. We decompose $f \in C^1(I ; \C^{\G_q} )$ as 
\begin{equation}
f = f_1 + \sum_{ 1 \neq q' | q} f_{q'}
\end{equation}
with $f_{q'} \in E^q_{q'}$. 
It is here that the finite bad modulus $Q_0$ of Theorem \ref{mainthm} enters. If we assume that $q$ has no proper divisors  $\leq q_0$ from Lemma \ref{powerssmallfreq2}, then for any $m$, with all norms $C^1$ norms,
\begin{align}
 \| L^{m}_{s,q} f -  L^{m}_{s,q} f_1 \| &\leq  \sum_{q_0 < q' | q} \| L^{m}_{s,q} f_{q'} \| \\
&=   \sum_{q_0 < q' | q} \sqrt{ | \G_q(q') | } \| L^{m}_{s,q'}  \Phi_{q,q'}^{-1} f_{q'} \|\\
& \leq  C \sum_{q_0 < q' | q} \sqrt{ | \G_q(q') | } (q')^C \rho^m \|  \Phi_{q,q'}^{-1} f_{q'} \| \\
&\leq   C q^C \rho^m \sum_{1 \neq q' | q} \|  f_{q'} \|.
\end{align}
This bound can be changed to
\begin{equation}
 \| L^{m}_{s,q} f -  L^{m}_{s,q} f_1 \|  \leq C' q^{C'} \rho^m \| f \| 
\end{equation}
for some $C' = C'(\G , \Q , b_0) $   by noting that individually
\begin{equation}
\| f_{q' } \| \leq \| f \|
\end{equation}
and that any number $q$ has $\ll_\e q^\e$ divisors for any $\e > 0$. The analogous estimates hold for $\L_{s,q}$ (by perturbation theory and \eqref{eq:perturbed}). That is, by possibly adjusting constants slightly and decreasing $a_0$
\begin{equation}
 \| L^{m}_{s,q} f -  L^{m}_{s,q} f_1 \|  \leq C' q^{C'} \rho^m \| f \|  .
\end{equation}
In particular part \ref{mainfirst} of Theorem \ref{mainthm} now follows from the case that $f_1 = 0$ so that $f \in C^1( I ;\C^{\G_q} \ominus 1)$.


\section{The modular group}\label{modular}
\def\S{\mathcal{S}}
\subsection{Features of the group}
In this section we give results about random walks in the modular group $\tilde{G} = \SL_2(\Z)$.
Firstly we have \cite{SERRETREES}
\begin{equation}\label{eq:amalg}
G \equiv \tilde{G}/ Z  = \mathrm{PSL}_2(\Z)  \cong \Z /2\Z \star \Z / 3\Z 
\end{equation}
where $\star$ stands for free amalgamated product and $Z$ is the center of $\tilde{G}$. Let $a$ and $b$ be generators for $G$ of order 2 and 3 respectively according to \eqref{eq:amalg}. We let $d$ denote the word distance between elements of $G$ with respect to these generators and will write $B_R(g)$ for the associated balls.
We will need the following fact about the free product of cyclic groups. The following can be deduced from \cite[pg. 209, Theorem 4.5]{MKS}.
\begin{lemma}\label{commutingelements}
The centralizer of any element $g \in G$ is either a conjugate of the $\Z / 2\Z$ or $\Z / 3\Z$ factor when $g$ is in that conjugate of that  factor, or an infinite cyclic group.
\end{lemma}
As a consequence of the Kurosh subgroup theorem \cite[Corollary 4.9.1]{MKS} one also has
\begin{lemma}
The only torsion elements of $G$ are conjugates of elements of the finite factors.
\end{lemma}

The following is a version of a Lemma of  Bourgain and Gamburd \cite[Lemma 3]{BOURGAINGAMBURDEXP} for a free product.
\begin{lemma}\label{largecomm}
Let $R\geq 2$. If $\S \subset B_{2R}( e )$ satisfies $|\S| \geq R^6$, then $|[\S,\S]| \geq R^3$.
\end{lemma}
\begin{proof}
Suppose for contradiction's sake  that $|[ \S , \S ]| < R^3$. Then we can find $A$ such that
\begin{equation}
\#\{  \{g_1, g_2\} \subset \S : [ g_1 , g_2 ] = A \} \geq | \S |^2 R^{-3}.
\end{equation}
We then also find a non-identity element $B\in \S$ such the set
\begin{equation}
\S_ 1  \equiv \{ g \in \S : [ g, B ] = A  \}
\end{equation}
satisfies
\begin{equation}\label{eq:S1bound}
| \S_1 | \geq  | \S | R^{-3}.
\end{equation}
Direct calculation gives that if $[ g , B] = [h ,B ] = A$ then the elements $B$ and $h^{-1} g$ commute. If $B$ is a torsion element then by Lemma \ref{commutingelements} this specifies $h^{-1}g$ up to at most three possibilities so that $\S_1$ has size at most $3$, which contradicts \eqref{eq:S1bound}.

Otherwise if $B$ is not torsion we can find non torsion $W$ generating the centralizer of $B$ by Lemma \ref{commutingelements}. Then $\S_1 \subset \langle W \rangle \cap B_{2R}(e)$. 
It then follows from \cite[Theorem 4.6]{MKS} that $W$ is conjugate by an element $h'$ of word length $\leq R$ to a non torsion cyclically reduced word $W^{h'}$ (this means that $W^{h'}$ is a reduced word in $a$ and $b$ which begins and ends with different letters). Then
\begin{equation}
(\S_1^{h'} )^{-1} \S_1^{h'} \subset  \langle W^{h'} \rangle \cap B_{4R}(e) 
\end{equation}
where we perform setwise multiplication and inversion.
As $W^{h'}$ is cyclically reduced and non torsion we have that $d( e ,  (W^{h'})^n ) = |n| d(e ,  W^{h'}) \geq 2|n|$ and so
\begin{equation}
| \S_1 |  = | \S_1^{h'} | \leq | (\S_1^{h'} )^{-1} \S_1^{h'} | \leq 4R + 1 .
\end{equation}
This leads to a contradiction with \eqref{eq:S1bound}.

\end{proof}

\begin{lemma}\label{vanishingcomm}
Let $R>2$. If $\S \subset B_{2R}(e)$ satisfies $ [[ \S , \S ] , [\S ,\S] ]  = \{ e \}$, then
\begin{equation}
| \S| < R^6 .
\end{equation}
\end{lemma}
\begin{proof}
Suppose that $| \S | \geq R^6$ for a contradiction. We know by Lemma \ref{largecomm} that if  we write $ \S^{(1) } = [ \S , \S]$ then $| \S^{(1)} | \geq R^3$. On the other hand the elements of $S^{(1)}$ commute so by Lemma \ref{commutingelements}  $\S^{(1)}$ is contained in a cyclic group. Since $|S^{(1)}| \geq R^3$ we know this group is not torsion and hence we can write each $s \in \S^{(1)}$ as a power of some fixed non torsion $W$. By the same arguments as in the proof of Lemma \ref{largecomm} this gives $|S^{(1) } | \leq 2R +1 $ which is a contradiction.
\end{proof}

\subsection{The random walk for almost square free $q$}

The following kind of argument is alluded to in \cite[pg. 275]{BGSACTA} although we have to work harder to
 cover the case when the values $c_0(I)$ generate a free semigroup $\G$, but not a free group. This is crucial for our desired application in Section \ref{continuedfractions}, where for some alphabets the group generated by the values of the cocycle $c_0(I)$ is all of $\SL_2(\Z)$.

We prove the following useful Lemma:
\begin{lemma}[Hitting small subgroups]\label{hitting}
Let $\k \ge 1$ and $D > 1$. Let $q$ be a modulus such that no prime occurs in $q$ with exponent $> \k$. Then for $q' | q$ with $q' \geq q^{1/D}$ and any subgroup $H$ of $\G_{q'}$ which projects to a proper subgroup of $\G_p$ for each $p | q'$, there is $c = c( \k , D )$ such that for any $\e > 0$ and suitable $C(\e)$
\begin{equation}
\nu_{q'}^{( \lceil c \log q\rceil ) }( a H ) \leq C(\e) q^\e .
\end{equation}

\end{lemma}
\begin{proof}
One key fact we will use is \cite[Theorem 3.3.4]{DSV}, which says any proper subgroup $J$ of $\mathrm{PSL}_2(p)$, $p$ prime,  with $|J| > 60$ has trivial second commutator, or derived length 2 in other words. 


Let $q_1$ be the product of the prime divisors $p$ of $q'$ where $H$ projects to a group of size $\leq 120$ in $\G_p$. Let $q_2$ be the product of the remaining primes dividing $q'$ (hence where the projection has trivial second commutator). Then $q_0 =  q_1q_2$ is the square free part of $q'$. Here $\pi_q$ stands for the projection  to $\mathrm{PSL}_2(\Z / q \Z )$.

We wish   to bound the number of walks/words in $S = c_0(I)$ of length $R$ which hit $aH$ modulo $q_2$. It is sufficient to do this in $G = \mathrm{PSL}_2(\Z)$.  Let $a$, $b$ be generators of $G$ of  order $2$ and $3$ respectively as before. We write $B_R(e)$ for the ball in the word metric induced from $a$, $b$ on $G$. Then $\pi(S) \subset B_L(e)$ for some $L$, where $\pi$ is the projection $\SL_2(\Z) \to \mathrm{PSL}_2(\Z)$.
The number of words of length $R$ in $S$ hitting $aH \bmod q_2$ is clearly bounded by $|B_{LR}(e) \cap \pi_{q_2}^{-1}(aH) |$ where we henceforth replace $aH$ with its image in $\mathrm{PSL}_2(\Z / q\Z)$. If $H$ projects to a proper subgroup of $\SL_2( \Z / p \Z)$ for each $p | q'$ then $\pi_p(H)$ is proper in $\mathrm{PSL}_2(\Z / p\Z)$. We have also used that $\pi(S) $ generate a free semigroup so that the evaluation map from words to $G$ is injective.

Note that 
\begin{equation}
|B_{LR}(e) \cap \pi_{q_2}^{-1}(aH) | - 1 \leq | B_{2RL}(e) \cap \pi_{q_2}^{-1}(H) |
\end{equation}
as for each $g \in B_{LR}(e) \cap \pi_{q_2}^{-1}(aH)$ we can form for any other $h  \in B_{LR}(e) \cap \pi_{q_2}^{-1}(aH) $ the product $g^{-1} h \in  \pi_{q_2}^{-1}(H) \cap B_{2LR}(e)$ and these are distinct for distinct $h$ and fixed $g$.

So a bound for the size of
\begin{equation}
\mathcal{S}: =   B_{2RL}(e) \cap \pi_{q_2}^{-1}(H)
\end{equation}
would be good enough. When $ R < C' \log q_2$ for some $C' = C'(L)$, the only element of $B_{2LR}(e)$ which reduces to the identity mod $q_2$ is the identity. Since $\pi_{q_2} (\mathcal{S} ) \subset H$ and the second commutator $H^{(2)} = \{ e \}$, we have
\begin{equation}
\mathcal{S}^{(2)} = \{ e \} 
\end{equation}
when $R < C' \log q_2$.
By Lemma \ref{vanishingcomm} this implies that
\begin{equation}
| \mathcal{S} | < (LR)^6  < (C'L\log q_2)^6 
\end{equation}

This is the kind of estimate we require. However it is not sufficient if $q_2$ is very small compared to $q_1 q_2$, since we need $R$ to grow and the requirement $R < C' \log q_2 $ becomes too strict.

Now we produce an estimate for $| \mathcal{S} |$ which suffices when $q_1$ is not too small. By the same arguments as before it is sufficient to bound the size of
\begin{equation}
\mathcal{T} =   B_{2RL}(e) \cap \pi_{q_1}^{-1}(H).
\end{equation}
As before, when $R < C' \log q_1 $,  each individual element of $H \bmod q_1$ has at most one preimage under $\pi_{q_1}$ in $B_{2LR}(e)$. Therefore given that for $p | q_1$ the projection of $H$ to $\mathrm{PSL}_2(\Z / p\Z)$ has size $\leq 60$
\begin{equation}
| \mathcal{T} | \leq | H \bmod q_1 | \leq 60^{ \omega(q_1) } \leq (60)^{C_2 \log(q_1) / \log \log(q_1)} = q_1^{C_3 / \log\log q_1 } ,
\end{equation}
where $\omega(q_1)$ is the number of distinct prime factors of $q_1$.

Let us now note that when we choose
\begin{equation}
R = \left\lceil  \frac{ C' } {10 D \k } \log q  \right\rceil \leq \left\lceil  \frac{ C' } {10 \k } \log q'  \right\rceil\leq \left\lceil \frac{C'}{10}( \log q_1 + \log q_2 ) \right\rceil
\end{equation}
at least one of our previous cases must be valid. If $q_1 \geq (q_0 )^{1/2} \geq q'^{1/2\k} $ then 
\begin{equation}
\nu_{q'}^{(R) }( a H ) \leq  q_1^{C_3 / \log \log q_1 } \leq  q^{C_3 / ( \log \log q -   \log 2 \k D)}.
\end{equation}
If $q_2  \geq (q_0 )^{1/2} \geq q'^{1/2\k}$ then
\begin{equation}
\nu_{q'}^{(R) }( a H )  \leq ( C_4 \log q_2 )^6 \leq C_5 (\log q)^6.
\end{equation}
Either way the result follows.
\end{proof}

\subsection{The mixing property for square free $q$}
Recall our mixing property Definition \ref{KSdfn}. We suppose now that we are given the semigroup $\G$
which is freely generated by $S$, the index set $\Q$ and the constant $c_1$. We are to find $c_2 , c_3 , c_4$ and $q_0$ such that the conclusions of Definition \ref{KSdfn} hold. Suppose then that we are given a complex valued measure $\mu$ on $\G_q$ with $\| \mu \|_1 < B$. Good spectral radius estimates for the operator $\mu \star$ will not be available if $\mu$ is for example supported on a small subgroup of $\G_q$. With this is mind, for $q' | q$ and $\mu$ a complex valued measure on $\G_q$ let us define, following \cite{BGSACTA},
\begin{equation}
\| \pi_{q'}(\mu ) \|_\infty \equiv \max_{ g \in \G_{q'} , H < \G_{q'} } | \mu | ( g H )
\end{equation}
where the maximum is taken over subgroups $H < \G_{q'}$ that project to proper subgroups of $\G_{q''}$ for each $q'' | q'$. 
In \cite{BGSACTA}, Bourgain, Gamburd and Sarnak give the following result ($\ell^2$-flattening Lemma).
\begin{thm}\label{flattening}  Let $q$ be square free. Given $\eta > 0$, there are $C>0$ and $\eta'>0$ such that the following hold.
Suppose $\mu$ is a complex valued measure on $\G_q$ with $\| \mu \|_1 < B$ and 
\begin{equation}
\| \pi_{q'}(\mu) \|_\infty < q^{-\eta} B
\end{equation}
for each $q' | q$ with $q' > q^{1/10}$. Then
\begin{equation}
\| \mu \star \phi \|_2 \leq C B q^{-\eta'} \| \phi \|_2
\end{equation}
for all $\phi \in E_q$, the new subspace of $l^2( \G_q )$.
\end{thm}
We wish to apply this to our measure $\mu$. Let $\k = 1$ and $D = 10$ in Lemma \ref{hitting} and let $c_2$ be equal to the $c(\k , D)$ which is provided.
Now if 
\begin{equation}
| \mu(x) | \leq B  q^{-c_1c_2} \nu_q^{( \lceil c_2 \log q \rceil )}(x) 
\end{equation}
we have for each $q' | q$ with $q' \geq q^{1/10}$ and $H$ with proper projection to each subgroup of $\G_{q'}$ that
\begin{equation}
|\mu_{q'} |( aH ) \leq B  q^{-c_1c_2} \nu_{q'}^{( \lceil c_2 \log q \rceil )}(aH)  \leq B q^{ -c_1 c_2} C(\e) q^\e
\end{equation}
for any $\e > 0$. Choose $\e = c_1c_2 / 2$  to get
\begin{equation}
|\mu_{q'}| ( aH ) \leq B  q^{-c_1c_2 / 3 } 
\end{equation}
for all square free $q$ bigger than a fixed number $q_0$.
We can now set $\eta = c_1c_2/3$ as valid in Theorem \ref{flattening} and let $c_3$ be the provided $\eta'$ and $c_4$ be the provided $C$. The $c_2 , c_3 , c_4$ and $q_0$ we have provided here from $c_1$ establish
 \begin{lemma}\label{sfmix}
Any free semigroup $\G \subset \SL_2(\Z)$ 
has the property (MIX) with respect to the set of all square free numbers.
\end{lemma}

\section{Dynamics associated to continued fractions}\label{continuedfractions}
\subsection{Continued fractions and Markov maps}
Let $\A$ denote a finite subset of $\mathbb{N}$ with at least two elements and recall the semigroups $\G_\A \subset {\mathcal G}_\A$ from the introduction, together with the generators $g_a$.
We will use $\mathcal{G}_\A$ to construct an $I_\A \subset \R$ and $T_\A : I_\A \to \R$ that satisfy the dynamical properties that are needed for our main Theorem \ref{mainthm}. We drop the $\A$ dependence from now on in this section.

We consider the $g_a$ acting as M\"{o}bius transformations in the upper half plane by
\begin{equation}
g_a(z) = \frac{1 }{z + a} .
\end{equation}
Let $A$ denote the largest member of $\A$. For $a \in \A$ let 
\begin{equation}
I_a = \left[   \frac{1}{ a + 1} ,   \frac{1}{ a+ (A+1)^{-1} }        \right] \subset \left[  \frac{1}{ a+ 1 } , \frac{1}{ a}          \right] .
\end{equation}
The $I_a$ are clearly disjoint as $A \geq 1$. Also note that
\begin{equation}
g_a \left( \left[ \frac{1}{A+1}  ,  1 \right] \right) = I_a .
\end{equation}
We define
\begin{equation}
I_0 \equiv \bigcup_{a \in A } I_a \subset  \left[ \frac{1}{A+1}  ,  1 \right].
\end{equation}
It follows that $g_a$ maps $I_{a'}$ into $I_a$ for all $a, a' \in \A$. We define $T_0 : I_0 \to \R$ by
\begin{equation}
T_0 \lvert_{I_a} = g_a^{-1} .
\end{equation}
Then $T_0$ clearly satisfies the Markov property and the corresponding symbolic dynamics is the full shift on the alphabet $\A$. By the Ping-Pong Lemma this construction of the $I_a$ also shows that the $g_a$ generate a free semigroup.

As we mentioned we wish to pass to the semigroup generated by products $g_a g_{a'}$. We set
\begin{equation}
I_{a , a'} = g_a I_{a'} \subset I_{a}
\end{equation}
giving a disjoint collection of closed intervals.
We set $I = \cup_{a , a' \in \A} I_{a , a'}$ and by noting $g_{a}g_{ a'} I \subset I_{a , a'}$ the map
\begin{equation}
T : I \to \R , \quad T\lvert_{I_{a , a'} } = (g_a g_{a'} )^{-1}
\end{equation}
has the Markov property with respect to the intervals $I_{a,a'}$. Set $K=\cap_{i=0}^\infty T^{-i}(I)$.

Note that the derivative of the matrix $g_a$ is 
\begin{equation}
g'_a(z)  =  \frac{1 } { | z + a |^2 } \leq  (a + (1+A)^{-1} )^{-2} \leq (1  + (1+ A)^{-1} )^{-2}
\end{equation}
when $z \in [ (1 + A )^{-1} , \infty )$. It follows that by use of the chain rule that for $z \in I$, there are some $a$ and $a' \in \A$ such that
\begin{equation}
| T'(z) | =  |[( g_a g_a' )^{-1} ]'(z) | \geq (1 + (1 + A)^{-1} )^4 > 1
\end{equation}
when $z \in I_{a, a'} = g_a g_{a'}( [ (1 +A)^{-1} , 1 ] )$. Let $\g := \g(\A) = (1 + (1 + A)^{-1})^4$. It now follows that for $z \in T^{-N +1}(I)$, we have
\begin{equation}
|[T^N]'(z) | \geq \g^N 
\end{equation}
so we have verified the expanding property for $T$. This gives a symbolic encoding of $\G_\A$ which is again the full shift, now on $|\A|^2$ letters.

We also note here that since $I_{a, a'} = g_a g_{a'}( [ (1 +A)^{-1} , 1 ] )$ it follows that $| T'(z) |$ is an analytic positive function on $I_{a,a'}$ which is bounded away from $0$ and is the restriction of a complex analytic function on a neighborhood of $I_{a, a'}$.

\subsection{The distortion function}
The distortion function $\tau$ is defined as 
\begin{equation}
\tau(x) \equiv  \log | T' (x)  |.
\end{equation}
Since $| T'(x) |$ is positive and has an analytic continuation to a neighborhood of the $I_{a , a'}$, it follows that by using the principal branch of the logarithm,
 $\tau$ has an analytic continuation to a neighborhood of $I$.

Recall the definition of the functions $\Delta_{\xi}$ and $\varphi_{\xi , \eta}$ from \eqref{eq:Delta} and \eqref{eq:temporaldistance}. It follows from the expanding property of $T$ (see \cite[pp. 129-130]{NAUD} for this implication) that if $\xi , \eta  \in \Sigma^-$ with $T(I_{\xi_0} ) \supset I_j$ and $T(I_{\eta_0} ) \supset I_j$ then $\varphi_{\xi , \eta}$ is real analytic on $I_j \times I_j$. 

We are going to show that $\tau$ has the non local integrability property (NLI) that we require. We say that $\tau$ is \textit{cohomologous} to $f$ on $K$ if there is a function $g$ such that
\begin{equation}
\tau(x) = f(x) + g(x) - g (T x) , \quad x \in K.
\end{equation}
We need the following Lemma from \cite[Lemma 4.3]{NAUD}. The Lemma appears in symbolic form in work of Dolgopyat \cite[Proposition 4]{DOLGPREVALENCE} and apparently goes back to work of Anosov. Recall that the cylinders of length $2$ are sets of the form $T^{-1}( I _{j_1} ) \cap I_{j_2}$  with $j_1$ and $j_2$ in $\mathcal{A}\times \mathcal{A}$. 
\begin{lemma}[Anosov alternative]\label{alternative}
Suppose that the transition matrix associated to $T$ and $I$ is symmetric\footnote{This is not required but Naud uses it.}. Then the temporal distance function $\varphi_{\xi , \eta}(u,v)$ is zero for all $j$, all $u,v \in I_j$ and all $\xi , \eta  \in \Sigma^-$ with $T(I_{\xi_0} ) \supset I_j$ and $T(I_{\eta_0} ) \supset I_j$ if and only if $\tau$ is cohomologous on $K$ to a function that is constant on cylinders of length $2$. 
\end{lemma}
As in our case the transition matrix is that of the full shift, the Lemma applies. Note from the definition \eqref{eq:temporaldistance} of $\varphi_{\xi, \eta}$ that $\varphi_{\xi ,\eta}(v,v) = 0$ for all $v \in I_j \subset T( I_{\xi_0} ) \cap T( I_{\eta_0} )$. If the property (NLI) did not hold for $\tau$, then for any $j , \eta , \xi$ as before we would have
\begin{equation}
\frac{\partial \varphi_{\xi , \eta }  }{\partial u } (u_0 , v_0) = 0
\end{equation}
for all $u_0 , v_0 \in K \cap I_j$. Since $\varphi_{\xi,\eta }$ is real analytic on $I_j \times I_j$ and $I_j \cap K$ contains accumulation points, it would follow that $\varphi_{\xi ,\eta}\equiv 0 $ on $I_j \times I_j$. Then Lemma \ref{alternative} would imply that $\tau$ is cohomologous on $K$ to a function that is constant on cylinders of length 2. We will now show that this cannot be the case. 

We need to recall the correspondence between periodic elements of $\Sigma^+$ and the traces and fixed points of corresponding group elements. This correspondence is well known in the setting of Fuchsian groups. We now introduce the notation $i \to j$ whenever $T(I_i) \supset I_j$. Here $i$ and $j$ are in the alphabet $\A \times \A$.

We write $\Sigma^p$ for the set of finite sequences of the form
\begin{equation}\label{eq:cycle}
j_0 \to j_1 \to j_2 \to j_3 \to \cdots \to j_l = j_0,
\end{equation}
which we call admissible cycles. Given an admissible cycle $c$ as in \eqref{eq:cycle} it follows that
\begin{equation}
\g_c = T_{j_0}^{-1} \circ T_{j_1}^{-1} \circ T_{j_2}^{-1} \circ \cdots \circ T_{j_{l-1}}^{-1} \in \G_\A
\end{equation}
maps $I_{j_0}$ to $I_{j_0}$ and hence by Brouwer's Theorem has a fixed point $q(c) \in I_{j_0}$. This $q(c)$ must be the attracting fixed point of the M\"{o}bius transformation $\g_c$. The fundamental fact that we will use is the formula
\begin{equation}\label{eq:tracetau}
| \tr( \g_c ) | = 2 \cosh \left( \frac{ \tau^l( q(c) ) }{2} \right) .
\end{equation}
In the setting of Fuchsian groups this is also related to the length of the closed geodesic corresponding to $\g_c$.

We make use of  the following fundamental trace identity

\begin{lemma}[Trace identity]\label{trace}
For all $g ,h \in \mathrm{SL}_2(\R)$ we have
\begin{equation}
\tr ( ghgh ) - \tr(g^2h^2) = \tr( h g h^{-1} g^{-1} ) - 2.
\end{equation}
\end{lemma}
\begin{proof}
Repeatedly use the identity
\begin{equation}
\tr( A B )  = \tr(A) \tr(B) - \tr(A B^{-1} )
\end{equation}
for $A , B \in \SL_2(\R)$. 
\end{proof}

\begin{prop}\label{distortionprop}
 The distortion function for the alphabet $\A$ has the non local integrability property.
\end{prop}

\begin{proof}
If $\tau$ does not have the non local integrability property, which we assume for a contradiction, then by Lemma \ref{alternative} and our argument from before $\tau$ is cohomologous to a function $f$ which is constant on cylinders of length $2$. That is
\begin{equation}
\tau(x) = g(x) - g(Tx) + f(x) 
\end{equation}
for some $g$. If $x$ is a periodic point of $T$ corresponding to a cycle $x \in \Sigma^p$, so that $T^n x  = x$ then this implies
\begin{equation}
\tau^n(x) = \sum_{ i = 0 }^{n-1} \tau( T^i x ) =\sum_{ i = 0 }^{n-1} g( T^i x )  - g( T^{i+1} x ) + f( T^i x ) = f^n(x).
\end{equation}
As $f$ is constant on cylinders of length $2$ it follows that when $x$ has period $n > 1$, $\tau^n (x)$ only depends on the multiset of pairs $(j_i , j_{i+1})$ which appear in the cycle $c$ given by
\begin{equation}
j_0 \to j_1 \to \cdots\to j_{n-1} \to j_n = j_0.
\end{equation}
Now let $c_1$ and $c_2$ be cycles of lengths $l_1 , l_2 \geq 2$ which begin and end at fixed $j_0$. Let $\g_1$ and $\g_2$ be the corresponding elements of $\G_\A$. Then the cycle $C_a = c_1 . c_2 . c_1 . c_2$ corresponds to the group element $\g_1\g_2 \g_1\g_2$ and the cycle $C_b = c_1 . c_1 . c_2 .c_2$ corresponds to $\g_1^2 \g_2^2$. Here, the symbol $.$ stands for the joining of cycles at their endpoint vertex $j_0$.
 Let  $L = 2l_1 + 2l_2$, the period of $C_a$ and $C_b$, and $q_a , q_b$ the fixed points in $I_{j_0}$ of corresponding group elements.
The pairs $(j_i , j_{i+1})$ which appear in $C_a$ and $C_b$ are the same when counted with multiplicity and so it follows that 
\begin{equation}
 \tau^L (  q_a ) = \tau^L(q_b) .
\end{equation} 
Now one  sees from \eqref{eq:tracetau} that
\begin{equation}
| \tr( \g_1 \g_2 \g_1 \g_2 ) |  =  | \tr(\g_1^2 \g_2^2 ) |.
\end{equation}
It is easy to see by direct calculation that the traces of elements of $\G_\A$ are positive and strictly greater than two.
Now Lemma \ref{trace} gives that $[ \g_1 , \g_2 ] = \g_1 \g_2 \g_1^{-1} \g_2^{-1}$ is a parabolic element of $\SL_2(\Z)$. We can ensure this does not happen and hence obtain a contradiction as follows.

We can choose $\g_1$ and $\g_2$ to have different attracting fixed points. This is possible since the periodic orbits of $T$ are dense in $K$ and $K$ has accumulation points in $I_{j_0}$. Note that $\g_1$ and $\g_2$ are loxodromic, meaning that they have two distinct fixed points on the ideal boundary of $\H$.  They  generate a discrete subgroup of $\SL_2(\Z)$. This implies that they do not have any fixed points in common, since they have different attracting fixed points and a group generated by a pair of distinct loxodromics with only one fixed point in common is not discrete.
Now by replacing  $\g_1$ and $\g_2$ with $\g_1^p$ and $\g_2^p$ for large enough $p$ depending on the separation of the 4 distinct fixed points of $\g_1$, $\g_2$, we can ensure $\g_1$ and $\g_2$ generate a Schottky sub\textit{group} of $\SL_2(\Z)$, in particular one that is freely generated by $\g_1 , \g_2$. In this case $[\g_1 , \g_2]$ cannot be a parabolic, leading to a contradiction.

\end{proof}

\subsection{The renewal equation on the boundary}\label{renewalsection}

We now show how one can adapt the work of Lalley \cite{LALLEYSYMB} to get counting estimates in our setting. The arguments of Lalley with the renewal equation do not produce an error term. However, given the strong bounds on the resolvent of the transfer operator that we have produced, it is possible to get a uniform error term in our counting problem. Most of the necessary arguments are given in \cite{BGSACTA}, however there is a small omission from their work which is the bridging between finite sequences, where the renewal equation applies to a counting problem, and infinite sequences where the strong bounds for the transfer operator hold. This bridging is carried out by Lalley in \cite[Theorem 4]{LALLEYSYMB} without any error term.
We work here to show that using the expanding property of the map $T$, this technical difficulty can be overcome. In some sense these arguments are the analog of a lemma of Ruelle from \cite{RUELLEFREDHOLM} relating the transfer operators to a dynamical zeta function. Consideration of the renewal equation offers an alternative framework to that of zeta functions.

We now adapt Lalley's work to our present framework. For simplicity, we write $\G=\Gamma_{\A}$ for the rest of this section.
Let $g \in C^1(I)$ be a non-negative function and $x \in I$. We define
\begin{equation}
N(a , x) =  \sum_{n = 0}^{\infty} \sum_{ y : T^n y = x} g(y) \mathbf{1}\{ \tau^n(y) \leq a \},
\end{equation}
where $\mathbf{1}\{ \tau^n(y) \leq a \}$ is the characteristic function of $\{ \tau^n(y) \leq a \}$. 
Only finitely many of the $n$ give a contribution to the sum, since $\tau$ is eventually positive. The \textit{renewal equation} states
\begin{equation}\label{eq:renewal}
N(a,x) = \sum_{ y : Ty = x } N(a - \tau(y) , y ) +g(x) \mathbf{1}\{a \geq 0 \}.
\end{equation}
This is related to the transfer operator $\L_{-s\tau}$ by taking a Laplace transform in the $a$ variable.
If one defines
\begin{equation}
n( s , x ) = \int_{-\infty}^\infty e^{ -s a } N(a , x) da
\end{equation}
then \eqref{eq:renewal} is transformed into
\begin{equation}
n(s , x ) =  [ \L_{-s \tau} n( s, \cdot) ] (x) + \frac{g(x)}{s} ,
\end{equation}
or
\begin{equation}
s n(s, x) = [(1 - \L_{-s\tau} )^{-1}  g ](x).
\end{equation}
The congruence version of the renewal equation at level $q$ concerns the quantity
\begin{equation*}
N_q(a , x , \vp ) \equiv \\  \sum_{n = 0}^{\infty} \sum_{ y : T^n y = x} g(y)  \rho (  c^n_q ( y )) \vp \mathbf{1}\{ \tau^n(y) \leq a \}.
\end{equation*}
where $\vp$ is a test function in $\C^{\G_q}$ and  we recall that $\rho$ is the right regular representation and $c^n_q $ is reduction mod $q$ of the cocycle defined in \eqref{defineourselvesacocyle}.
The congruence renewal equation reads
\begin{equation}
N_q(a , x , \vp ) =\sum_{ y : Ty = x } \rho( \pi_q ( c_0 ( y ) ))   N_q(a - \tau(y) , y , \vp)  +  g (x)  \vp   \mathbf{1}\{ 0 \leq a \} .
\end{equation}
so that the same arguments from before give
\begin{equation}\label{eq:nqdef}
s n_q(s, x , \vp) = [(1 - \L_{-s\tau, q} )^{-1}  g \otimes \vp ](x)
\end{equation}
where $g \otimes \vp$ is the vector valued function taking $x \mapsto   g (x)  \vp $ and 
\begin{equation}
n_q( s , x ,\vp ) = \int_{-\infty}^\infty e^{ -s a } N_q(a , x , \vp) da .
\end{equation}
Notice that $N_q$ and hence $n_q$ is linear in $\vp$. We can split into two cases as we can also write
\begin{equation}
\vp = \vp_0 + \vp' 
\end{equation}
where $\vp_0$ is constant and $\vp'$ is orthogonal to constants. The analysis of $N_q(a ,x ,\vp_0)$ boils down to that of $N(a, x)$, which is in principle understood without any of the results of this paper. We take up the analysis in the case that
\begin{equation} \vp' \in \C^{\G_q} \ominus 1,\end{equation} that is, orthogonal to constants. Assume this is the case from now on.

Then under the hypothesis of Theorem \ref{mainthm} (we have established the non local integrability property (NLI) for $\tau$ and will assume property (MIX) holds), we can estimate $(1 - \L_{-s\tau ,q})^{-1} g \otimes \vp'$ since $g \otimes \vp' \in C^1( I ; \C^{\G_q} \ominus 1)$.

One obtains from \eqref{eq:nqdef} and Theorem \ref{mainthm} that for any $\eta > 0$
\begin{equation}\label{eq:nqbound}
|s | \| n_q(s , \bullet , \vp' ) \|_{C^1} \leq \begin{cases} C q^C (1 - \rho_0)^{-1}\| g \otimes \vp \|_{C^1} \text{ if } |b| \leq b_0 \\
C_\eta | b|^{1 +\eta } ( 1 - \rho_\eta )^{-1}\| g \otimes \vp \|_{C^1}  \text{ if } |b| > b_0
\end{cases}
\end{equation}
with the same quantifiers and constants as in Theorem \ref{mainthm}. Consolidating constants, for any $\eta >0$ there is $C' = C'(\eta)$ such that
\begin{equation}\label{eq:nqdecay}
|s | \| n_q(s , \bullet, \vp' ) \|_{C^1}  \leq C' \max( q^C  , |b|^{1 + \eta} ) \| g \otimes \vp \|_{C^1}
\end{equation}
whenever $|a - s_0| < \e$ for some sufficiently small $\e$.

 We also note that given the bounds in Theorem \ref{mainthm}, it follows that the correspondence
\begin{equation}
s \mapsto (1 - \L_{-s\tau , q})^{-1} g \otimes \vp'
\end{equation}
gives a holomorphic family of $C^1$ functions in the region $|a  - s_0| < \e$ for fixed $g$ and $\vp'$, hence $n_q(s , x ,\vp')$ is holomorphic for $s$ in this region. 
Recall that we have $s_0=\delta_{\mathcal A}=\delta$.
This is essential for the contour shifting argument to follow. Now we follow technical work of Bourgain, Gamburd and Sarnak \cite[pp. 25-26]{BGSACTA} to extract information about $N_q(a,x ,\vp')$.

Let $k$ be a smooth nonnegative function on $\R$ such that
\begin{equation}
\int k = 1,
\end{equation}
\begin{equation}
\mathrm{support}(k) \subset \left[ 1 , 1 \right],
\end{equation}
and
\begin{equation}
|\hat{k }( \xi ) | \leq B \exp( - |\xi|^{1/2} )
\end{equation}
for some $B$, where
\begin{equation}
\hat{k}(\xi) \equiv \int_\R e^{-\xi t} k(t) dt.
\end{equation}

Then let for small $\lambda > 0$
\begin{equation}
k_\lambda(t) = \lambda^{-1} k ( t \lambda^{-1} ),
\end{equation}
this has the effect that
\begin{equation}\label{eq:scaledfourier}
\hat{k_\lambda}(\xi) = \hat{k}( \lambda \xi) , \quad |\hat{k_\lambda}(\xi)| \leq B \exp( - |\lambda \xi| ^{1/2}).
\end{equation}
Consider the smoothed quantity of interest
\begin{equation}
\int_{-\infty}^{\infty} k_\lambda(t)  N_q ( a + t, x , \vp'  ) dt = \frac{1}{2\pi i } \int_{s \in \delta + i\R} e^{as} n_q(s , x ,\vp' ) \hat{k_\lambda}(s ) ds.
\end{equation}
by inverting the Laplace transform and interchanging the order of integration. From \eqref{eq:nqdecay}, $n_q$ is well enough behaved that this is possible. For technical reasons let $\e' = \min( \delta / 2 , \e/2)$.
 We can shift the contour to $\Re(s) = \delta - \e'$ to get that the above is the same as
\begin{align}
&\frac{1}{2\pi i } \int_{s \in \delta - \e' + i\R} e^{as} n_q(s , x ,\vp' ) \hat{k_\lambda}(s ) ds \\
&= \frac{1}{2\pi} e^{a(\delta -\e')}  \int_{\theta \in \R} e^{ ai\theta} n_q(\delta -\e' + i\theta , x ,\vp' ) \hat{k_\lambda}(\delta -\e' + i\theta) d\theta 
\end{align}
where $s = \delta - \e' + i\theta$. Putting in the bound \eqref{eq:nqbound} for $n_q$ together with \eqref{eq:scaledfourier} gives the new bound
\begin{align}
&\frac{BC'}{2\pi} e^{a(\delta -\e')} \| g \otimes \vp' \|_{C^1 }   \left(q^C \int_{| \theta | \leq b_0} | \delta -\e' + i\theta|^{-1}   e^{ -|\lambda (\delta -\e' + i\theta)|^{1/2} } d\theta  \right.\\
&+ \left. \int_{ |\theta | > b_0 }| \delta -\e' + i\theta|^{-1}   |\theta |^{1 +\eta}  e^{ -|\lambda (\delta-\e' + i\theta)|^{1/2} } d\theta \right) \\
& \leq \frac{BC'}{2\pi} e^{a(\delta -\e')} \| g \otimes \vp' \|_{C^1 }  \left(\frac{4 q^C  b_0}{\delta -\e'} + C'' |\lambda |^{-1 -\eta} \right)
\end{align}
for some new absolute constants $C', C''$.
Putting this together (choosing $\eta = 1$ is enough) gives
\begin{lemma}\label{congtermrenewal}
Suppose that $\G$ has property (MIX) for $q \in \Q$. There is $Q_0>0$ provided by Theorem \ref{mainthm} and positive
constants $\e'$, $C$, $\k_1 $, $\k_2$ such that for $ q \in \Q$ with $(Q_0, q) = 1$ and any $g \in C^1(I) $, $\vp' \in \C^{\G_q} \ominus 1$ we have
\begin{equation}
\| \int_{ - \lambda }^{\lambda} k_{\lambda}(t)N_q(a + t , x , \vp' ) dt \| < e^{a(\delta -\e')} \| g \otimes \vp' \|_{C^1 }\left( \k_1 q^C + \k_2 |\lambda|^{-2} \right)
\end{equation}
where the norm on the left hand side is the one in $\C^{\G_q}$.

\end{lemma}

We now describe $N_q(a, x , \vp_0)$ with $\vp_0$ a constant function. In this case the counting reduces to the non congruence setting. The following is a straightforward adaptation of \cite[Proposition 10.2]{BGSACTA} to our setting. This is an effectivization of work of Lalley \cite{LALLEYSYMB}, using the work of Naud \cite{NAUD} as input to get a power saving error term. Let  $\underline{ 1 }$ be the constant function in $\C^{\G_q}$ taking on the value 1.
\begin{lemma}\label{maintermrenewal}
There exists $\e''>0$ such that for any $q$, $g \in C^1(I) $ we have 
\begin{equation}
\int_{-\lambda}^{\lambda} k_\lambda(t) N_q( a + t , x , \underline{1} ) dt = C(x , g) e^{\delta a}\underline{1}  + O(\| g \|_{C^1} | \G_q | \lambda^{-3} e^{ ( \delta - \e'')a }),
\end{equation}
where 
\begin{equation}
C(x ,g) = \left(\frac{ \int g d\nu_{-\delta \tau} }{ \delta \int \tau d \nu_0 } \right) h_{- \delta \tau}(x).
\end{equation}
is a $C^1$ function of $x$ and the error is estimated in $C^1$ norm, and $\nu, h$ are the measures and functions coming from the Theorem \ref{RPFTheorem}. 
\end{lemma}
We remark that the $| \G_q | \| g \|_{C^1}$ in the error term above comes from $\| g \otimes \vp_0 \|_{C^1}$. We can now put these Lemmas together to get
\begin{prop}\label{boundaryprop}
Suppose that $\G$ has property (MIX) for $q \in \Q$. There exists $Q_0>0$ provided by Theorem \ref{mainthm} such that when $ q \in \Q$ with $(Q_0,q) = 1$, the following holds. There is  $\e > 0$ such that for any non negative $\vp \in \R^{\G_q} \subset \C^{\G_q}$,
\begin{equation}
 N_q(a , x , \vp) = \frac{ C(x , g) e^{\delta a}  \langle \vp , \underline{ 1 }  \rangle  \underline {1} }{ | \G_q |}+ O \left( e^{( \delta  - \e  ) a }  q^C \| g\|_{C^1} \| \vp \| \right) 
\end{equation}
where  $\langle  \cdot ,\cdot \rangle$ is the standard inner product.
\end{prop} 
\begin{proof}
Decompose $\vp$ as
\begin{equation}
\vp = \frac{\langle \vp , \underline{ 1 }  \rangle  \underline {1} }{ | \G_q |} + \vp' .
\end{equation}
Then Lemmas \ref{congtermrenewal} and \ref{maintermrenewal} give that
\begin{align}
\int_{-\lambda}^{\lambda} k_\lambda(t) N_q(a +t , x , \vp) dt&= \frac{ C(x , g) e^{\delta a}  \langle \vp , \underline{ 1 }  \rangle  \underline {1} }{ | \G_q |} \\
&+ e^{a(\delta - \e )} O \left(  \| g \|_{C^1} \| \vp \| ( \k_1 q^C + \k_2  \lambda^{-2} +  \lambda^{-3} ) \right)
\end{align}
by using that
\begin{equation}
\| g \otimes \vp' \|_{C^1 } \leq \| \vp' \| \| g \|_{C^1 }
\end{equation}
and replacing $\e' , \e''$ with a new small enough $\e$. Now taking $\lambda = e^{  - a \e / 6 }$ we have that the error term is
\begin{equation}
e^{a (\delta - \e/2 ) } O(q^C \| g\|_{C^1} \| \vp \| ).
\end{equation}
Since $\vp$ is non negative,  $N_q(a , x , \vp )$ is increasing in $a$ and hence
\begin{equation}
N_q( a - \lambda , x, \vp )\leq  \int_{-\lambda}^{\lambda} k_\lambda(t) N_q(a +t , x , \vp) dt \leq N_q(a + \lambda , x ,\vp) \end{equation}
which is enough to get the result given the exponentially shrinking $\lambda$, by replacing $\e$ with some smaller value.
\end{proof}

\subsection{Proof of main Corollary \ref{maincoro}}
So far we have given good bounds for the quantity $N(a , x , \vp)$.
		 In practice however, it is not this quantity that one is interested in but the related $\R^{\G_q}$ valued function
\begin{equation}
N^*_q(a , \g_0 , \vp )  \equiv  \sum_{  \g \in \G \cup \{1\}  \: : d( o, \g \g_0 o) - d( o , \g_0 o) \leq a } G(\g \g_0 o) \rho( \pi_q (\g) ). \vp 
\end{equation}
where 
\begin{itemize}
\item $G$ is a non negative function on $\H \cup \R$ with the property that there exist an integer $M$ and neighborhood $J_M$ of the length $M$ cylinders in $I$ such that $G$ is constant on $J_M$. We write $g$ for the restriction of $G$ to $\R$.
\item $\vp \in \R^{\G_q}$ ,  $\pi_q : \G \to \G_q$ is reduction mod $q$ and $\rho$ is the right regular representation of $\G_q$.
\item $o\in \H$ is a fixed origin and $\g_0 \in \G$.
\end{itemize} We now show how to relate the quantities $N^*_q$ and $N_q$. Let $d_E$ denote Euclidean distance in the upper half plane. We need to note the following Lemma.
\begin{lemma}\label{contracting}
There is some $r > 0$ such that $T$ has analytic extension to a neighborhood
\begin{equation}
T : I^{(r)} = \{ x \in \H \cup \R  \: : \: d_E(x , I ) < r \} \to \H \cup \R .
\end{equation}
Moreover, there is $0<\k < 1$ such that if $x \in I^{(r)}$ then 
\begin{equation}\label{eq:contracting}
d_E( g_i x , I ) \leq \k d_E( x  , I )
\end{equation}
for all $i  \in \A \times \A$. Here $g_i$ is a generator of $\G_\A$.
\end{lemma} \begin{proof}
The fact that $T$ has a complex  analytic extension to a neighborhood of $I$ is clear. The inequality \eqref{eq:contracting} follows from the fact that M\"{o}bius transformations preserve circles and the expanding property of $T$ proved earlier.
\end{proof}
 Let $\G^{(n)}$ denote those $\g \in \G$ which are a product of at least $n$ generators. Define for $n \geq 1$ the shift
\begin{equation}
\sigma : \G^{(n)} \to \G^{(n - 1)} , \quad \sigma( g_{i_1} g_{i_2} \ldots g_{i_n} ) = g_{i_2} \ldots g_{i_n} 
\end{equation} with the convention that $\G^{(0) } = \G \cup  \{ e \}$ and $\sigma(g_i) = e$ for all $i \in \A \times \A$.
Define for $\g \in \G$
\begin{equation}
\tau_*( \g  ) = d( o ,  \g o ) - d( o ,  (\sigma \g) o ) .
\end{equation}
Define for $n \geq N$ and $\g \in \G^{(n) }$
\begin{equation}
\tau_*^N(\g ) = \sum_{j = 0}^{N-1} \tau_*(\sigma^j \g )  =  d( o ,  \g  o ) - d( o ,  (\sigma^N \g) o ) .
\end{equation}
We can now recast $N^*_q$ as
\begin{equation}
N_q^*(a, \g_0 , \vp ) = \sum_{n=0}^\infty \sum_{\g \in \G : \sigma^n \g = \g_0 } G( \g ) \rho( \pi_q(\g \g_0^{-1}) )\cdot  \vp   \mathbf{1}\{ \tau_*^n (\g ) \leq a \} .
\end{equation}
One obtains again a renewal equation:
\begin{equation}\label{eq:finiterenewal}
N_q^*(a, \g_0 , \vp ) = \sum_{ \g : \sigma \g = \g_0 }  N_q^*(a - \tau_*(\g) , \g , [\rho( \pi_q( \g \g_0^{-1}) )\vp] ) + G( \g_0 ) \vp \mathbf{1}\{a \geq 0 \}  
\end{equation}
where $\rho$  is the right regular representation.

\begin{lemma}\label{bridging}
Fix $k_0 \in K$.
Let $\g \in \G^{(n)} $ and $\g_0 \in \G^{(N)}$. Then 
\begin{equation}
\tau_*^n( \g \g_0 ) = \tau^n( \g \g_0 k_0 ) + O ( \k^N ).
\end{equation} 
\end{lemma}
\begin{proof}
There is some $n_0$ such that all $\g \in \G^{(n_0)}$, $\g o \in I^{(r)} $. Then by Lemma \ref{contracting}, for all $N > n_0$ we have for all $\g \in \G$ that
\begin{equation}\label{eq:convergence}
d_E( \g \g_0 o , \g \g_0 k_0 ) \ll  \k^{n + N - n_0  } .
\end{equation}
with implied constant depending only on $\A$, and $\k$ as in Lemma \ref{contracting}. We also have
\begin{equation}
\tau_*(  g_{j_0} g_{j_1} g_{j_2} \ldots g_{j_{n-1}} \g_0 ) =- \log | g_{j_0} '( g_{j_2} \ldots g_{j_{n-1}} \g_0 o) | + o( d_E(  g_{j_2} \ldots g_{j_N} \g_0 o , \R) ).
\end{equation}
A similar estimate is given in \cite[pg. 41]{LALLEYSYMB}. It follows then that
\begin{equation}
\tau_*(  g_{j_0} g_{j_1} g_{j_2} \ldots g_{j_{n-1}} \g_0 )  =- \log | g_{j_0} '( g_{j_2} \ldots g_{j_{n-1}} \g_0 o) | + O( \k^{n + N -1 - n_0} ).
\end{equation}
Since there is some uniform bound for the  derivative of $\log| [g_i]' |$ close to $I$, this together with \eqref{eq:convergence} implies 
\begin{equation}
\tau_*(  g_{j_0} g_{j_1} g_{j_2} \ldots g_{j_{n-1}}\g_0 )  = - \log | g_{j_0} '( g_{j_2} \ldots g_{j_{n-1}} \g_0 k_0) |  + O(\k^{ n + N  - 1 - n_0} ) .
\end{equation}

By iterating and summing the geometric series it follows that 
\begin{equation}
\tau_*^n(  g_{j_0} g_{j_1} g_{j_2} \ldots g_{j_{n-1}}\g_0 ) = - \log   |  ( g_{j_0}  g_{j_2} \ldots g_{j_{n-1} } )' \g_0 k_0 ) | + O(\k^{N  -n_0} ) 
\end{equation}
or what is the same,
\begin{equation}\label{eq:bridging}
\tau_*^n( \g \g_0 )  = \tau^n (\g \g_0 k_0) +  O(\k^{N- n_0 } ).
\end{equation}
We can absorb $n_0$ into the implied constant, proving the Lemma.
\end{proof}

\begin{lemma}\label{sandwich}
Fix some $k_0 \in K$ and suppose $\vp$ is non negative.
There are $N_0$, $\k < 1$ and $C$ depending on $G$ such that if $\g_0 = g_{j_0} g_{j_1} g_{j_2} \ldots g_{j_N}$ with $N > N_0$ we have $\g_0 \in I^{(r)} $ and 
\begin{equation} 
 N_q(a - C \k^N , 
\g_0 k_0  , \vp ) \leq N_q^*(a, \g_0   , \vp ) \leq  N_q(a + C \k^N , \g_0 k_0  , \vp ) .
\end{equation}
 The inequalities are between functions in $\R^{\G_q}$.
\end{lemma}
\begin{proof}
We have for fixed $k_0 \in K$ and $\g_0 = g_{j_0} g_{j_1} g_{j_2} \ldots g_{j_N}$ a one-to-one correspondence
\begin{equation}
k: \g \to I ,\quad k( \g  ) = \g  k_0 
\end{equation}
with the property that $\sigma^n \g = \g_0$ if and only if $T^n k(\g) = \g_0 k_0 = k(\g_0 ) $. Under this correspondence one has
\begin{align}\label{eq:finiteN}
N_q^*(a, g_{j_0} g_{j_1} g_{j_2} \ldots g_{j_N}   , \vp ) = \sum_{n=0}^\infty \sum_{\substack { \g \in \G : \sigma^n \g = \g_0 } } G( \g )  \rho(\pi_q(\g\g_0^{-1}))\vp  \mathbf{1}\{ \tau_*^n (\g ) \leq a \}
\end{align}
and
\begin{align}
N_q(a, \g_0 k_0 , \vp ) =\sum_{n=0}^\infty \sum_{\substack { \g \in \G : \sigma^n \g = \g_0 }} g( k (\g) )  \rho(\pi_q(\g\g_0^{-1}))\vp \mathbf{1}\{ \tau^n ( k(\g)  ) \leq a \}.
\end{align}
These can now be compared term by term. If $N$ is large enough, depending on $G$, then $G(\g) = g( k(\g))$ for all terms as all the $\g$ will lie in the neighborhood $J_M$. On the other hand, we have from Lemma \ref{bridging} that if $\sigma^n \g = \g_0$
\begin{equation}
\mathbf{1}\{ \tau_*^n(\g) \leq a \} \leq \mathbf{1}\{ \tau^n(k(\g)) \leq a + C \k^{N - n_0 } \}
\end{equation}
for some $C$ and $n_0$
and
\begin{equation}
 \mathbf{1}\{ \tau^n(k(\g)) \leq a - C \k^{N - n_0} \} \leq \mathbf{1}\{ \tau_*^n(\g) \leq a \} .
\end{equation}
Given that $\vp$ and hence $\rho(\pi_q(\g\g_0^{-1} ))\vp$ are positive functions, inserting these inequalities into \eqref{eq:finiteN} gives the result, by increasing $n_0$ to $N_0$ so that the condition $\g \in J_M$ also holds when $\sigma^n = \g$ for some $n \geq 0$.
\end{proof}

Following Lalley \cite[pg. 22]{LALLEYSYMB} we iterate the finite renewal equation \eqref{eq:finiterenewal} to obtain
\begin{align}
N_q^*(a, \g_0 , \vp ) &= \sum_{ \g : \sigma^n \g = \g_0 } N_q^*(a - \tau_*^n(\g), \g , \rho[ \pi_q ( \g \g_0^{-1} ) ] \vp ) \\
&+ \sum_{m=1}^{n-1} \sum_{\g : \sigma^m \g = \g_0 } G( \g ) \rho[ \pi_q( \g \g^{-1}_0) ]\vp  \mathbf{1}\{a  - \tau_*^m(\g) \geq 0 \}  + G( \g_0 ) \vp\mathbf{1}\{a \geq 0 \}.
\end{align}
We want to increase $n$ so we note that the second line is bounded by
\begin{equation}
\sum_{m = 0 }^{n-1} | \A \times \A |^m \| G \|_\infty \| \vp \|  \ll  \| G  \|_\infty \| \vp \|  | \A \times \A |^n.
\end{equation}
We will take $n = \lfloor c  a \rfloor$ for small enough $c$.
 This gives
\begin{equation}
N_q^*(a, \g_0 , \vp ) = \sum_{ \g : \sigma^n \g = \g_0 } N_q^*(a - \tau_*^n(\g), \g ,  \rho[ \pi_q ( \g \g_0^{-1} )] \vp ) + O ( \| G \|_\infty \| \vp \| e^{2 c \log |\A|  a} ) .
\end{equation}
We can now use Lemma \ref{sandwich} to get that up to  $O( \| G \|_\infty \| \vp \|  e^{2 c \log |\A|  a} )$, $N_q^*(a , \g_0 , \vp)$ is sandwiched between 
\begin{equation}
\sum_{ \g : \sigma^n \g = \g_0 } N_q(a - \tau_*^n(\g) - C\k^n, \g k_0,  \rho[ \pi_q ( \g \g_0^{-1} ) ] \vp ) 
\end{equation}
and
\begin{equation}
 \sum_{ \g : \sigma^n \g = \g_0 } N_q(a - \tau_*^n(\g) + C\k^n, \g k_0 ,  \rho[ \pi_q ( \g \g_0 ^{-1}) ] \vp ) .
\end{equation}
Using the precise asymptotics of Proposition \ref{boundaryprop}, under the hypothesis that $q \in \Q$ with $(Q_0, q) =1$ where $\G$ has property (MIX) we have that
\begin{align}
N_q^*(a, \g_0 , \vp ) &= \left(1 + O( \delta C \k^n ) \right) \frac{e^{\delta a }}{| \G_q |}   \langle \vp , \underline{ 1 }  \rangle  \underline {1}   \sum_{\g : \sigma^n \g = \g_0 } C(\g k_0 , g)  e^{ - \delta \tau_*^n(\g) } \\
&+ O \left( q^C \| g\|_{C^1} \| \vp \|e ^{( \delta  - \e  ) a }  \sum_{\g : \sigma^n \g = \g_0 } e^{-( \delta  - \e  )  \tau_*^n(\g) }  \right)   + O( \| G \|_\infty \| \vp \| e^{2 c \log |\A|  a} ).
\end{align}
Given that $n = \lfloor c a \rfloor$ for some small $c$ yet to be chosen, the $\k^n$ term will not be significant. We do however have to describe the terms
\begin{equation}
 \sum_{\g : \sigma^n \g = \g_0 } C(\g k_0 , g)  e^{ - \delta \tau_*^n(\g) }
\end{equation}
and 
\begin{equation}
 \sum_{\g : \sigma^n \g = \g_0 } e^{-( \delta  - \e  )  \tau_*^n(\g) } .
\end{equation}
The latter can be bounded using Lemma \ref{bridging} with $N = 0$ to give $\tau_*^n( \g) = \tau^n(\g k_0 ) + O(1)$ and hence 
\begin{equation}\label{eq:taubridged}
 \sum_{\g : \sigma^n \g = \g_0 } e^{-( \delta  - \e  )  \tau_*^n(\g) } \ll \sum_{ k : T^n k = \g_0 k_0 } e^{-(\delta - \e) \tau^n( k ) } = [\L_{-(\delta - \e) }^n 1 ](\g_0 k_0 ).
\end{equation}
We know that $\L_{- (\delta - \e )\tau}$ is bounded by $\exp( P( - (\delta - \e) \tau ))$ by the Ruelle-Perron-Frobenius theorem. We now therefore require $n< \frac{a \e}{2 P( - (\delta - \e ) \tau ) }  $ so that
\begin{equation}
 [\L_{-(\delta - \e) }^n 1 ](\g_0 k_0 ) \leq \exp( n P( - (\delta - \e) \tau ) ) \leq \exp( a \e / 2 ).
\end{equation}
To describe the main term
\begin{equation}\label{eq:mainterm}
\frac{e^{\delta a } }{ | \G_q | }   \langle \vp , \underline{ 1 }  \rangle  \underline {1} \sum_{\g : \sigma^{n} \g = \g_0 }    C(\g k_0 , g) e^{-\delta \tau_*^n(\g) } ,
\end{equation}
we require the following result of Lalley (cf. \cite[Theorem 4]{LALLEYSYMB}). It says that there is a version of the maximal eigenfunction $h_{-\delta \tau}$ on $\G$, as opposed to $K$.
\begin{lemma}\label{hinside}
Fix $k_0 \in K$.
There is a  unique positive function $h_* : \G \to \R$ such that there is $\theta > 1$ so that  if $\g \in \G^{(n)}$
\begin{equation}
h_* ( \g )  = h_{ - \delta \tau } ( \g k_0 ) + O( \theta^{- n } ) .
\end{equation}
Also, for all $\g \in \G$,
\begin{equation}\label{eq:stationaryhstar}
h_*(\g) = \sum_ {\g' : \sigma(\g') = \g } e^{-\delta \tau_*(\g') } h_*(\g').
\end{equation} 
\end{lemma}

Now recall the definition of $C( \cdot, g)$  from Lemma \ref{maintermrenewal}. If we define the corresponding function on $\G$ according to the pairing of $h_*$ with $h_{-\delta \tau}$,
\begin{equation}
C_*( \g , g ) = \left(\frac{ \int g d\nu_{-\delta \tau} }{ \delta \int \tau d \nu_0 } \right) h_{*}(\g) ,
\end{equation}
we get from Lemma \ref{hinside} that
\begin{equation}
C_*( \g , g) = C( \g k_0 , g) + O( \| g\|_{C^1 }\theta^{-n} )
\end{equation}
when $\g \in \G^{(n) }$.
This means that the main term contribution \eqref{eq:mainterm} to $N_q^*(a , \g_0 , \vp)$ is
\begin{align}
&\frac{e^{\delta a } }{ | \G_q | }   \langle \vp , \underline{ 1 }  \rangle  \underline {1} \left( \sum_{\g : \sigma^{n} \g = \g_0 }    C_*(  \g, g) e^{-\delta \tau_*^n(\g) }  + O(\|g \|_{C^1 } \theta^{-n}\sum_{\g : \sigma^{n} \g = \g_0 }   e^{-\delta \tau_*^n(\g) } ) \right) \\
& =  \frac{e^{\delta a } }{ | \G_q | }   C_*( \g_0, g)  \langle \vp , \underline{ 1 }  \rangle  \underline {1}  + e^{\delta a } O( \theta^{-n} \| \vp \| \|g \|_{C^1})
\end{align}
by using \eqref{eq:stationaryhstar}  and a calculation similar to that in \eqref{eq:taubridged} to give
\begin{equation}
\sum_{\g : \sigma^{n} \g = \g_0 }   e^{-\delta \tau_*^n(\g) } \ll [\L^n_{-\delta} 1] (\g_0 k_0 ) \leq 1 .
\end{equation}
We now let  $n  = \lfloor c a \rfloor $ with
\begin{equation}
c = \min\left( \frac{ \delta - \e }{4 \log | \A | } ,  \frac{\e }{ 2 P( -(\delta -\e) \tau ) } \right) .
\end{equation}
 Then the result of the preceding discussion is that
\begin{equation}
N_q^* ( a , \g_0 , \vp ) = \frac{e^{\delta a } }{ | \G_q | } C_*(\gamma_0, g)  \langle \vp , \underline{ 1 }  \rangle  \underline {1} + O\left( (\|\vp\|  (\|g \|_{C^1} + \| G \|_\infty ) q^C e^{(\delta - \e' ) a }\right) 
\end{equation}
for some $\e' = \e'( \k , \theta , \e , \A    )$ . 
This proves our main Corollary \ref{maincoro} given the following observations.
When $\vp( \g ) = \mathbf{1}\{ \g = \xi \}$ we have that
\begin{equation}
\langle \vp , \underline 1 \rangle  = 1
\end{equation}
and hence, under the mixing hypothesis on $\G$, evaluating $N_q^*(a, \g_0 ,  \mathbf{1}\{ \g = \xi \})$ gives 
\begin{equation}
\sum_{  \substack {\g \in \G_\A  \: : d( o, \g \g_0 o) - d( o , \g_0 o) \leq a \\ \pi_q(\g) = \xi }} G(\g \g_0 o)    = \frac{e^{\delta a } }{ | \G_q | }  C_*(\gamma_0, g) + O\left( ( \|g \|_{C^1} + \| G \|_\infty ) q^C e^{(\delta - \e' ) a }\right) .
\end{equation}
 In addition, we have the identity 
\begin{equation}
\| \g \|^2 =  2\cosh( d ( o , \g o) )
\end{equation}
when $o$ is chosen to be $i \in \H$. Then the condition $ d( o, \g \g_0 o) - d( o , \g_0 o) \leq a$ becomes
\begin{equation}
\frac{ \| \g \g_0 \|}{ \| \g_0 \| } \leq R ,
\end{equation}
where $R  = \sqrt{2\cosh( a )} = e^{a/2} + O(e^{-3a/2 } )$. With these remarks one obtains our main Corollary \ref{maincoro}.


\end{document}